\documentclass[11pt]{amsart}

\usepackage{amsmath,amsfonts,amssymb,latexsym,url,graphicx}
\usepackage{tikzsymbols}
\usepackage[dvipsnames]{xcolor}
\usepackage{hyperref}
\usepackage{txfonts,pifont,bbding,pxfonts}
\usepackage{manfnt}
\usepackage[active]{srcltx}
\usepackage{cases}
\usepackage{wasysym,pstricks,enumerate,subfigure}
\usepackage{coffeestains}
\usepackage{lscape,color}
\usepackage{soul}
\usepackage[none]{hyphenat}
\usepackage[mathscr]{euscript}
\usepackage{wrapfig}
\usepackage{enumitem}
\usepackage{tikz}
\usepackage{pgfplots}
\usepackage[numbers]{natbib}

\pgfplotsset{compat=1.18}

\usepackage[T1]{fontenc}
\usepackage{xcolor} 

\usepackage{booktabs}

\newcommand{\R}{{\mathbb R}} 

\newtheorem{theorem}{Theorem}[section]
\newtheorem{corollary}[theorem]{Corollary}
\newtheorem{lemma}[theorem]{Lemma}
\newtheorem{proposition}[theorem]{Proposition}
\newtheorem{definition}[theorem]{Definition}
\newtheorem{remark}[theorem]{Remark}

\usepackage[x11names,svgnames]{xcolor}

\newcommand{\edits}[1]{\textcolor{Blue}{#1}}

\usepackage{lineno}
\usepackage{mathtools}

\usepackage{tikz}
\usepackage{circuitikz}
\usetikzlibrary{3d}
\usepackage{float}
\title{Functional Ordinary Differential Equations and the Design of Dichromatic Lenses}
\author[]{Mohammad Tarek Al Masri, St\'ephane Najjar, Ahmad Sabra*, Maya Said,\\ Muna Sattouf }
\thanks{*\, Corresponding Author}
\address{Department of Mathematics \\ American University of Beirut\\ Beirut, Lebanon}

\thanks{2020 AMS Math Subject Classification: 78A05, 78A46, 93C23.}
\email{maa369@mail.aub.edu, sfn15@mail.aub.edu, asabra@aub.edu.lb, mks52@mail.aub.edu, munasattouf@gmail.com}
\date{\today}

\begin{document}
\maketitle

\definecolor{newpurple}{RGB}{158, 84,197}

\begin{abstract}
This paper establishes an existence and uniqueness theorem for a class of functional ordinary differential equations under assumptions weaker than those found in the literature. We then consider the problem of designing a lens that refracts rays of two distinct frequencies emitted from a point source into a common collimated beam. We derive a system of functional ordinary differential equations whose solvability characterizes the existence of such lenses. Using the existence theorem, we prove that such a lens exists in two dimensions. We then extend the corresponding result to three dimensions.

\end{abstract}
    

\tableofcontents

\section{Introduction}

Inverse problems in Geometric Optics aim to construct optical systems with prescribed reflective or refractive properties. Such systems are either designed for  imaging purposes, where the goal is to create a certain image \cite{SIAMGUTSAB}, \cite{Verma2026}, or for non-imaging aims, which are more concerned on the direction of the light leaving the optical system \cite{Gutierrez:14} 
and on the illumination distribution received at the target \cite{Oliker2019}, \cite{GutierrezHuang2009}, \cite{GutierrezHuang2014}. Such inverse problems have attracted considerable mathematical research and have led to the study of nonlinear partial differential equations  including Monge–Amp\`ere type equations \cite{MerigotThibert2021}, \cite{guan1998monge}, \cite{CaffarelliOliker2008}, and in certain settings systems of functional differential equations \cite{FriedmanMcLeod1987}, \cite{rogers1988existence}, \cite{van1992lens}, \cite{van1994mathematical}, \cite{GutierrezSabra2020}.

One of the fundamental difficulties in  optical design is that the refractive index of a material depends on the  frequency of light \cite{JenkinsWhite2001}. Consequently, rays of different colors  refract at different angles when passing from one medium to another, causing chromatic aberration. In practice, this effect is usually corrected by the use of non-conventional materials or metasurfaces specifically engineered to control the frequency-dependent refraction \cite{Gutierrez:18}, \cite{Chen2018}, \cite{GutierrezSabra2021}, or by combining several lens elements made from different materials, such as achromatic doublets \cite{SunChuTien2009}, \cite{Hua2017}. 
This motivates the question we study in this paper which is whether a homogeneous single-element lens can be designed to refract rays emitted from a point source and composed of two frequencies into a common collimated beam.

The design of such a lens can be formulated as an inverse problem. Under suitable assumptions, in two dimensions, the unknown refracting faces of the lens satisfy a system of nonlinear functional differential equations (FDEs) of the form
$$F(t,Z(t),Z(Z_1(t)), Z'(t),Z'(Z_1(t)))={\bf 0}.$$
Unlike the classical delay or neutral FDEs \cite{HaleLunel1993}, the systems of interest in this paper belong to a class of equations in which the unknown functions are arguments of other unknown functions. This 
prevents the direct application of standard existence and uniqueness results from the theory of FDEs.

  FDEs for inverse problems in Geometric Optics first appeared in \cite{FriedmanMcLeod1987}, then in   \cite{rogers1988existence}, \cite{van1992lens}, and more recently in \cite{GutierrezSabra2020}. These papers aim to design two dimensional lenses that perfectly focus rays composed of multiple wavelengths into one  point in the near field case or into a common direction in the far field case. 

The existence theory in \cite{rogers1988existence} and \cite{GutierrezSabra2020} relies on the construction of a norm under which the Lipschitz constants associated with two derivative variables satisfy a contraction condition. Controlling both constants is a challenge in this case and is not always possible \cite[Corollary 4.6]{GutierrezSabra2020}. In this paper, we prove an existence and uniqueness result that is coordinate free and does not depend on the choice of such a norm, Theorem \ref{thm:Main Theorem}. We then revisit the dichromatic lens problem in \cite{GutierrezSabra2020} and show that it can be formulated as a system of three FDEs, as opposed to five equations in \cite{GutierrezSabra2020}. This reduction, together with Theorem \ref{thm:Main Theorem}, yields a stronger existence result for dichromatic lenses in two dimensions by enlarging the class of admissible configurations for which local solutions exist, see Remark \ref{rmk:admissibility of p} and \cite[Remark 5.14]{GutierrezSabra2020}.

The paper is organized as follows. In Section \ref{sec: system of functional ODEs}, we establish an existence and uniqueness result for a system of FDEs of the form \eqref{eq:system} using a fixed-point argument and the implicit function theorem. This result is of independent mathematical interest and could be applicable beyond the dichromatic problem considered here. After a preliminary review of Snell's law in Section \ref{subsub:Snell} and of the monochromatic problem \cite{Gutierrez2013} in Section \ref{prelim:monochromatic}, we introduce the dichromatic problem in Section \ref{sec:Setup}. We show in Theorems \ref{eq:dichromatic theorem} and \ref{thm:converse} that the existence of solutions to the 2D-dichromatic problem is equivalent to the existence of solutions to an FDE of the form studied in Section \ref{sec: system of functional ODEs}, see Definition \ref{def:F_i}. We then apply Theorem \ref{thm:Main Theorem} in Section \ref{sec:existence} and show the existence of a solution to the two dimensional dichromatic problem. We in fact show that the obtained lens is symmetric which enables us to conclude, in Section \ref{subsec:3D}, the existence of a lens solving the dichromatic problem in three dimensions. Finally, in Section \ref{subsec:numerical}, we implement a numerical construction of the solution lens.

\section{System of Functional Differential Equations}\label{sec: system of functional ODEs}
In this section, we will use the following notation.
\begin{itemize}
    \item For $x\in \mathbb R^n$, we write $x=(x_1,\cdots,x_n)$ with $x_i\in \mathbb R$ the $i-$th coordinate of $x$.
    \item For $U$ open in $\mathbb R^{m+n}$, and $H:U\mapsto \mathbb R^n$ a $C^1$ map with $H(x,y)=(H_1(x,y),\cdots, H_n(x,y))$, $x=(x_1,\cdots,x_m)\in \mathbb R^m$ and $y=(y_1,\cdots,y_n)\in \mathbb R^n$; $\nabla_{y}H=\left( \frac{\partial H_j}{\partial y_k}\right)_{1\leq j,k\leq n}$ denotes the $n\times n$ matrix derivative of $H$ with respect to $y$, and $DH=\begin{pmatrix}\nabla_xH|\nabla_yH\end{pmatrix}$ denotes the total derivative of $H$.
    \item Let $\mathcal{M} \in M_{n \times n}(\mathbb R)$. The spectral radius $\; R_{\mathcal{M}} = \max\{ \, |\lambda| : \lambda$ is an eigenvalue of $\mathcal{M}\, \}$.
    \item Let $\|\cdot\|$ be a norm in $\mathbb R^n$, the induced matrix norm on $M_{n\times n}(\mathbb R)$ is defined as follows $$\left|\|A\|\right|=\max\{\|Av\|: v\in \mathbb R^n, \|v\|=1\},\qquad \qquad A\in M_{n\times n}(\mathbb R).$$
   \end{itemize}

Let $F$ be a map defined in an open domain $U \subseteq \mathbb{R}^{4n+1}$ with values in $\mathbb{R}^n$. We write $$F:=F(X) = (F_1(X), F_2(X), \cdots , F_n(X))$$
with $X = (t, \alpha,\beta,\zeta,\xi) \in U,$ $t\in \mathbb{R}$, and $\alpha,\beta,\zeta,\xi\in \R^n$. 
The goal of this section is to study the existence and uniqueness of $C^1$ solutions $Z(t)=(Z_1(t),Z_2(t),\cdots,Z_n(t))$ defined in a neighborhood of $t=0$ to the following system of FDEs
\begin{equation}\label{eq:system}
    \begin{cases}
    F(t, Z(t), Z(Z_1(t)),Z'(t),Z'(Z_1(t))) = {\bf 0}\\
    Z(0) = {\bf 0}
    \end{cases},
\end{equation}
where ${\bf 0}$ denotes the zero vector in $\R^n$.
\begin{remark}\label{rmk:necessary condition}
  Notice that if $Z$ solves \eqref{eq:system} in $[-t_0,t_0]$, then $(t,Z(t),Z(Z_1(t)),Z'(t),Z'(Z_1(t)))$ must belong to $U$ for every $t$ and $F(0,{\bf 0}, {\bf 0}, Z'(0),Z'(0))={\bf 0}$. Therefore, a necessary condition for the solvability of \eqref{eq:system} is the existence of a vector ${\bf p}\in \R^n$ such that $(0,{\bf 0,0,p,p})\in U$ and
$F(0,{\bf 0,0},{\bf p},{\bf p})={\bf 0}.$   
\end{remark}
With the above setting, we show the main result of this section.
\begin{theorem}\label{thm:Main Theorem}
Let ${\bf p}=(p_1,p_2,\cdots,p_n)\in \mathbb R^n$ be such that $P:=(0,{\bf 0}, {\bf 0}, {\bf p}, {\bf p})\in U$ and $F(P)={\bf 0}$. Assume $F$ is $C^1$ in a neighborhood of $P$ contained in $U$. If
\begin{enumerate}
\item $|p_1|<1$,
\item  $\nabla_{\zeta}F(P)$ is invertible,
\item The spectral radius $R_{-[\nabla_{\zeta}F(P)]^{-1}\nabla_{\xi}F(P)}<1$,
\end{enumerate}
then the system \eqref{eq:system} has a unique $C^1$ solution $Z$ in a neighborhood of $t=0$ with $Z'(0)={\bf p}$.
\end{theorem}

We proceed through several steps. 

\paragraph{\bf Step 1.} 

Since $\nabla_{\zeta}F(P)$ is invertible, by the Implicit Function Theorem, 
there exists a neighborhood $V\subseteq U$ of 
$P$, a 
neighborhood $\mathcal O$ of ${\bf p}$ in $\mathbb R^n$, a neighborhood  $\mathcal W$ of $(0,{\bf 0},
{\bf 0}, {\bf p})$ in $\mathbb R^{3n+1}$, and 
a $C^1$ map $H: \mathcal W\mapsto \mathcal O$ such that
for every $(t,\alpha,\beta,\xi)\in \mathcal W$, 
\begin{equation}\label{eq:Implict}
F(t,\alpha,\beta,H(t,\alpha,\beta,\xi),\xi)={\bf 0}.
\end{equation}
Moreover, for every $(t,\alpha,\beta,\zeta,\xi) \in V$ satisfying $F(t,\alpha,\beta,\zeta,\xi)={\bf 0}$, $(t,\alpha,\beta,\xi)\in \mathcal W$ and $\zeta=H(t,\alpha,\beta,\xi)\in \mathcal O.$ 

We show in this step that, locally, \eqref{eq:system} is equivalent to the system 
\begin{equation}\label{eq:simplified system}
\begin{cases}
Z'(t)=H(t,Z(t),Z(Z_1(t)),Z'(Z_1(t)))\\
Z(0)={\bf 0}
\end{cases}
\end{equation}

In fact, if $Z$ is a $C^1$ solution to \eqref{eq:simplified  system} in a neighborhood of $t=0$ with $Z'(0)={\bf p}$, then from \eqref{eq:Implict}, $Z$ solves \eqref{eq:system}.

Conversely, assume $Z$ solves \eqref{eq:system} with $Z'(0)={\bf p}$, then there exists $t_0>0$ such that
$(t,Z(t),Z(Z_1(t)),Z'(t),Z'(Z_1(t)))\in V$ for $t\in [-t_0,t_0]$. From the uniqueness of the implicit function satisfying \eqref{eq:Implict}, we deduce that $Z$ solves \eqref{eq:simplified system}.\\
\paragraph{\bf Step 2.} In this step, we introduce a particular norm on $\mathbb R^n$ and derive some Lipschitz properties for the map $H$. We use the notation $\hat {P}=(0,{\bf 0, 0, p})\in \mathcal W$. 

Differentiating \eqref{eq:Implict} with respect to $\xi$, we get that
    $$\nabla_{\xi}H(\hat P)=-[\nabla_{\zeta}F(P)]^{-1}\nabla_{\xi}F(P).$$

From assumption $(3)$ of the Theorem, we obtain that $R_{\nabla_{\xi}H(\hat P)}<1$, then by Householder's theorem \cite[Chapter 7]{Serre}
, there exists a norm $\|\cdot \|_0$ on $\mathbb R^n$ such that 

\begin{equation}\label{eq:norm<1}
\left|\left\|\nabla_{\xi}H(\hat {P})\right\|\right|_0<1. 
\end{equation}

For $\varepsilon>0$, define the closed ball
$$B_{\varepsilon}(\hat {P})=\left\{\hat X=(t,\alpha,\beta,\xi):|t|+\|\alpha\|_0+\|\beta\|_0+\|\xi-{\bf p}\|_0\leq \varepsilon\right\}.$$
We choose $\varepsilon$ small enough so that $B_{\varepsilon}(\hat{P})\subseteq \mathcal W$.

Since $H\in C^1(B_{\varepsilon}(\hat {P}))$, there exists a constant $L>0$ such that
\begin{equation} \label{Lipschitz H}
\left\|H(\bar{t},\bar{\alpha},\bar{\beta},\xi)-H(t,\alpha,\beta,\xi)\right\|_0\leq L \left(|\bar t-t|+\left\|\bar{\alpha}-\alpha\right\|_0+\left\|\bar {\beta}-\beta\right\|_0\right)\end{equation}
for every $ (\bar{t},\bar{\alpha},\bar{\beta},\xi),(t,\alpha,\beta,\xi) \in B_{\varepsilon}(\hat{P}).$ Moreover we have the following Proposition.

\begin{proposition}\label{prop:Contraction}
For all $(t,\alpha,\beta,\bar\xi)$ and $(t,\alpha,\beta,\xi)$ in $B_{\varepsilon}(\hat {P})$

$$
\left\|H\left(t,\alpha,\beta,\bar{\xi}\right)-H\left(t,\alpha,\beta,\xi\right)\right\|_0\leq \left(\max_{\hat X\in B_{\varepsilon}(\hat {P})} \left|\left\|\nabla_{\xi}H({\hat X})\right\|\right|_0\right)
\left\| \bar{\xi}-\xi\right\|_0.
$$
\end{proposition}

\begin{proof}

Let $(t,\alpha,\beta,\bar{\xi}),(t,\alpha,\beta,\xi)$ in $ B_\varepsilon(\hat P)$, by the Fundamental Theorem of Calculus 
    \begin{align*}
        H(t,\alpha,\beta,\bar\xi) - H(t,\alpha,\beta,\xi)  & =\int_0^1DH\left((1-s)(t,\alpha,\beta,\xi) +s(t,\alpha,\beta,\bar{\xi})\right)(0,\mathbf{0},\mathbf{0},\bar\xi-\xi)^tds\\
        & = \int_0^1\nabla_{\xi}H\left((1-s)(t,\alpha,\beta,\xi) +s(t,\alpha,\beta,\bar{\xi})\right)(\bar{\xi}-\xi)^t\, ds.
    \end{align*}
Therefore,
    \begin{align*}
        \|H(t,\alpha,\beta,\bar{\xi}) - H(t,\alpha,\beta,\xi)\|_0&\leq \int_0^1\left\|\nabla_{\xi}H\left((1-s)(t,\alpha,\beta,\xi) +s(t,\alpha,\beta,\bar{\xi})\right)(\bar{\xi}-\xi)^t\right\|_0\, ds
        \\ &\leq \int_0^1\left|\left\|\nabla_{\xi}H\left((1-s)(t,\alpha,\beta,\xi) +s(t,\alpha,\beta,\bar{\xi}\right)\right\|\right|_0\, \|\bar \xi-\xi\|_0 \, ds\\
        & \leq \left(\max_{\hat{X}\in B_\varepsilon(\hat P)}\left\|\left|\nabla_{\xi} H(\hat X)\right\|\right|_0\right)\|{\bar \xi} - \xi\|_0.
    \end{align*}
\end{proof}

\begin{corollary}\label{cor:Lip}
There exists $\varepsilon_0>0$ for which the following hold
\begin{itemize}
\item For every $\hat X\in B_{\varepsilon_0}(\hat{P})$, 
\begin{equation}\label{eq:bound on h1}
|H_1(\hat X)|\leq 1.
\end{equation}
\item There is $C<1$ such that for every $\left(t,\alpha,\beta,\bar{\xi}\right), \left(t,\alpha,\beta,\xi\right) \in B_{\varepsilon_0}(\hat P)$ 
\begin{equation} \label{Lipschitz xi1}
\left\|H\left(t,\alpha,\beta,\bar{\xi}\right)-H\left(t,\alpha,\beta,\xi\right)\right\|_0\leq C \left\|\bar{\xi}-\xi\right\|_0.
\end{equation}
\end{itemize}
\end{corollary}

\edits{}

\begin{proof} 
We have that $|H_1(\hat{P})|=|p_1|<1,$ so by the continuity of $H$, \eqref{eq:bound on h1} follows in a neighborhood of $\hat{P}$.

The existence of $C<1$ satisfying \eqref{Lipschitz xi1} follows from the fact that $H$ is $C^1$, Proposition \ref{prop:Contraction}, and inequality \eqref{eq:norm<1}.
\end{proof}

Notice that the choice of the norm $\|\cdot\|_0$ guarantees that the map $H$ is a contraction in the variable $\xi$ in a neighborhood $B_{\varepsilon_0}(\hat{ P}).$
\hfill\\


\paragraph{\bf Step 3.} For $t_0>0$, we denote by $C^1([-t_0,t_0])$ the space of $C^1$ functions on $[-t_0,t_0]$ with the following norm
 $$
\|Z\|_{C^1([-t_0,t_0])}=\max_{t\in [-t_0,t_0]} \|Z(t)\|_0+\max_{t\in [-t_0,t_0]} \|Z'(t)\|_0.$$
In this step, we define a contraction map $T$ on a subset of $C^1([-t_0,t_0])$.

 Let $M=\max_{\hat X \in B_{\varepsilon_{0}}(\hat {P})} \|H(\hat X)\|_0$ and $K=\dfrac{L(1+2M)}{1-C}$ where $L$ and $C$ are the Lipschitz constants in \eqref{Lipschitz H} and \eqref{Lipschitz xi1}. We define the set ${\mathbb X_{t_0}}$ as follows: $Z\in \mathbb X_{t_0}$ if and only if
 
\begin{enumerate}[label=\roman*)]
    \item $Z\in C^1([-t_0,t_0])$ with $Z(0)={\bf 0}$ and $Z'(0)={\bf p}$,
    \item $\|Z(\bar t)-Z(t)\|_0\leq M |\bar t-t|$ for all $\bar t,t\in [-t_0,t_0]$,
    \item $|Z_1(\bar t)-Z_1(t)|\leq |\bar t-t|$ for all $\bar t,t\in [-t_0,t_0]$,
    \item $\|Z'(\bar t)-Z'(t)\|_0\leq K|\bar t-t|$ for all $\bar t,t\in [-t_0,t_0]$.
\end{enumerate}

Notice that $\mathbb X_{t_0}$ is not empty. In fact, since $H(\hat P)={\bf p}$ we have $\|{\bf p}\|_0\leq M$.
Together with $|p_1|<1$, this shows that $Z(t)=t{\bf p}\in \mathbb X_{t_0}$. Moreover, being  a closed subset of $C^1([-t_0,t_0])$, $\mathbb X_{t_0}$  is complete. 

Also, notice that  i) and iii) imply that 
 \begin{equation}\label{eq:z1 contract}
 |Z_1(t)|\leq |t|
 \end{equation}
 for every $t\in [-t_0,t_0]$ and $Z\in \mathbb X_{t_0}$.

We denote $V_{Z}(t)=(t,Z(t),Z(Z_1(t)),Z'(Z_1(t)))$.

\begin{proposition} \label{Vz in Ball}
Given $t_0>0$ such that $t_0\leq \dfrac{\varepsilon_0}{1+2M+K}$ and $Z\in \mathbb X_{t_0}$, then $V_Z(t)\in B_{\varepsilon_0}(\hat {P})$ for every $t\in [-t_0,t_0]$.
\end{proposition}

\begin{proof}
From $\mathrm i)$ and $\mathrm{ii})$, we have that for $t\in [-t_0,t_0]$
$$\|Z(t)\|_0\leq M |t|,$$ and so from \eqref{eq:z1 contract}, $Z_1(t)\in [-t_0,t_0]$ and $$\|Z(Z_1(t))\|_0\leq M |Z_1(t)|\leq M |t|.$$ Similarly from i), iv), and \eqref{eq:z1 contract}
$$\|Z'(Z_1(t))-{\bf p}\|_0=\|Z'(Z_1(t))-Z'(0)\|_0\leq K |Z_1(t)|\leq K |t|.$$
Therefore, for every $t\in [-t_0,t_0]$
$$
|t|+\|Z(t)\|_0+\|Z(Z_1(t))\|_0+\|Z'(Z_1(t))-{\bf p}\|\leq |t|+M |t|+M|t|+ K |t|\leq t_0(1+2M+K)\leq \varepsilon_0.
$$
\end{proof}

For $t_0\leq \dfrac{\varepsilon_0}{1+2M+K}$ and $Z\in \mathbb X_{t_0}$, we define the map $T:\mathbb X_{t_0}\mapsto C^1([-t_0,t_0])$
$$(TZ)(t)=\int_{0}^t H(V_Z(s))\, ds.$$
This map well defined since $H\in C^1(B_{\varepsilon_0}(\hat{P}))$, and 
 $V_Z(t)\in B_{\varepsilon_0}(\hat P)$ for $t\in [-t_0,t_0]$.

\begin{proposition}\label{contraction}
There exists $t_0\leq\dfrac{\varepsilon_0}{1+2M+K}$ positive such that $T$ is a contraction on $\mathbb X_{t_0},$ that is, $T(\mathbb X_{t_0})\subseteq \mathbb X_{t_0}$ and for some $q<1$ 
\begin{equation}\label{eq:contraction}
\left\|(TZ^1)-(TZ^2)\right\|_{C^1([-t_0,t_0])}\leq q \left\|Z^1-Z^2\right\|_{C^1([-t_0,t_0])},
\end{equation}
for every $Z^1,Z^2\in \mathbb X_{t_0}$.
\end{proposition}

\begin{proof}
Let $W=TZ$, with $Z\in \mathbb X_{t_0},$ we show that $W\in \mathbb X_{t_0}$.

From the definition of $T$, $W\in C^1([-t_0,t_0])$,  $W(0)={\bf 0}$, and $W'(0)=H(V_Z(0))=H(0,{\bf 0,0,p} )={\bf p}$.

For $ \bar t,t\in   [-t_0,t_0]$, say $t\leq \bar t$,  
\begin{align*}
    \|W(\bar t)-W(t)\|_0\leq \int_{ t}^{\bar t}\|H(V_Z(s))\|_0\, ds\leq M|\bar t-t|.
\end{align*}
Also, from inequality \eqref{eq:bound on h1}
$$|W_1(\bar t)-W_1(t)|\leq \int_t^{\bar t}|H_1(V_Z(s))|\, ds\leq |\bar t-t|.$$

It remains to check that $W$ satisfies iv).  From \eqref{Lipschitz H}, \eqref{Lipschitz xi1}, \eqref{eq:z1 contract},  the expression of $K$, and the fact that $Z\in \mathbb X_{t_0}$, we get that
\begin{align*}
    \left\|W'(\bar t)-W'(t)\right\|_0&=\|H(V_Z(\bar t))-H(V_Z(t))\|\\
    &\leq L (|\bar t-t|+\|Z(\bar t)-Z(t)\|_0+\|Z(Z_1(\bar t))-Z(Z_1(t))\|_0)+C\|Z'(Z_1(\bar t))-Z'(Z_1(t))\|_0\\
    &\leq L(|\bar t-t|+M |\bar t-t|+M|Z_1(\bar t)-Z_1( t)|) + CK |Z_1(t)-Z_1(\bar t)|\\
    &\leq \left(L(1+2M)+CK\right)|\bar t-t|
    =K|\bar t-t|.
\end{align*}
Therefore, $W\in \mathbb X_{t_0}.$

Next, we prove \eqref{eq:contraction}. Let $W^1=TZ^1$, and $W^2=TZ^2$ with $Z^1,Z^2\in \mathbb X_{t_0}.$ By the Fundamental Theorem of Calculus, for every $Z\in C^1([-t_0,t_0])$ satisfying $Z(0)={\bf 0}$
$$\|Z(t)\|_0=\left\|\int_0^tZ'(s)\, ds\right\|_0\leq \left(\max_{t\in [-t_0,t_0]}\|Z'(t)\|_0\right) t_0.$$
Then, 
\begin{equation}\label{eq:relation function derivative}
    \max_{t\in [-t_0,t_0]} \|Z(t)\|_0\leq t_0 \max_{t\in [-t_0,t_0]}\|Z'(t)\|_0\leq  t_0\|Z\|_{C^1([-t_0,t_0])}.
\end{equation}
Therefore, 
\begin{equation}\label{eq:W Lip}
\|W^1-W^2\|_{C^1([-t_0,t_0])}\leq (1+t_0) \max_{t\in [-t_0,t_0]}\|(W^1)'(t)-(W^2)'(t)\|_0.
\end{equation}

From the definition of $W^i$ and the Lipschitz properties in \eqref{Lipschitz H} and \eqref{Lipschitz xi1},
\begin{align*}
\|(W^1)'(t)-(W^2)'(t)\|_0&=\|H(V_{Z^1}(t))-H(V_{Z^2}(t))\|_0\\
&\leq L\left(\|Z^1(t)-Z^2(t)\|_0+\|Z^1(Z_1^1(t))-Z^2(Z_1^2(t))\|_0\right)+C\|(Z^1)'(Z_1^1(t))-(Z^2)'(Z_1^2(t))\|_0.
\end{align*}
We now estimate the terms in the above inequality. From \eqref{eq:relation function derivative},
$$\|Z^1(t)-Z^2(t)\|_0\leq t_0 \|Z^1-Z^2\|_{C^1([-t_0,t_0])}.$$
Since all norms in 
$\mathbb R^n$ are 
equivalent, there 
exists a constant 
$C_{\|\cdot\|_0}$ such 
that 
\begin{equation}\label{eq:norms}
|x_1|\leq 
C_{\|\cdot\|_0}\|x\|_0
\end{equation}
for every $x=
(x_1,x_2,\cdots,x_n)\in
\mathbb R^n$. Hence, 
using ii), \eqref{eq:z1 contract}, \eqref{eq:relation function derivative}, and \eqref{eq:norms}
\begin{align*}
    \|Z^1(Z_1^1(t))-Z^2(Z_1^2(t))\|_0&\leq \|Z^1(Z_1^1(t))-Z^2(Z_1^1(t))\|_0+\|Z^2(Z_1^1(t))-Z^2(Z_1^2(t))\|_0\\
    &\leq \max_{t\in [-t_0,t_0]}\|Z^1(t)-Z^2(t)\|_0 + M|Z^1_1(t) - Z^2_1(t)|\\
    &\leq \max_{t\in [-t_0,t_0]}\|Z^1(t)-Z^2(t)\|_0 + M C_{\|\cdot\|_0}\|Z^1(t)-Z^2(t)\|_0\\
    &\leq t_0(1+M C_{\|\cdot\|_0})\|Z^1-Z^2\|_{C^1([-t_0,t_0])}.
\end{align*}
Similarly, from iv), \eqref{eq:z1 contract}, \eqref{eq:relation function derivative}, and \eqref{eq:norms}
\begin{align*}
\|(Z^1)'(Z_1^1(t))-(Z^2)'(Z_1^2(t))\|_0&\leq \|(Z^1)'(Z_1^1(t))-(Z^2)'(Z_1^1(t))\|_0+\|(Z^2)'(Z_1^1(t))-(Z^2)'(Z_1^2(t))\|_0\\
&\leq  \max_{t\in [-t_0,t_0]}\|(Z^1)'(t)-(Z^2)'(t)\|_0+K|Z_1^1(t)-Z_1^2(t)|\\
&\leq \max_{t\in [-t_0,t_0]}\|(Z^1)'(t)-(Z^2)'(t)\|_0+K C_{\|\cdot\|_0}\|Z^1(t)-Z^2(t)\|_0\\
&\leq (1+K C_{\|\cdot\|_0}t_0)\|Z^1-Z^2\|_{C^1([-t_0,t_0])}.
\end{align*}
Combining the above inequalities, we conclude that
$$
\|(W^1)'(t)-(W^2)'(t)\|_0\leq
\left( L t_0(2+M C_{\|\cdot\|_0})+C(1+K C_{\|\cdot\|_0}t_0)\right)\|Z^1-Z^2\|_{C^1([-t_0,t_0])}.
$$
Hence, replacing in \eqref{eq:W Lip}
$$
\|W^1-W^2\|_{C^1([-t_0,t_0])}\leq (1+t_0) \left(Lt_0(2+M C_{\|\cdot\|_0})+C(1+K C_{\|\cdot\|_0}t_0)\right)\|Z^1-Z^2\|_{C^1([-t_0,t_0])}
$$
Notice that, as $t_0\to 0$, the term multiplying $\|Z^1-Z^2\|_{C^1([-t_0,t_0])}$ converges to $C<1$, so the Proposition follows for $t_0$ small enough.
\end{proof}

{\bf Step 4.} From Step 3, we choose $t_0$ small enough so that $T$ is a contraction on the complete metric space $\mathbb X_{t_0}$. By the Banach Fixed-Point Theorem, $T$ has a unique fixed point $Z^0\in \mathbb X_{t_0}$.

From the definition of $T$, $Z^0$ solves the system \eqref{eq:simplified system} on $[-t_0,t_0]$, and by Step 1, $Z^0$ solves \eqref{eq:system}. To complete the proof of the Theorem, we show that $Z^0$ is the unique solution to \eqref{eq:system} on $[-t_0,t_0]$ satisfying $(Z^0)'(0)={\bf p}$ (by making $t_0$ smaller if needed). Let $W^0\in C^1([-t_0,t_0])$ be a solution to \eqref{eq:system} with $(W^0)'(0)={\bf p}$. From Step $1$, $W^0$ is a local solution to \eqref{eq:simplified system} with $(W^0)'(0)={\bf p}$. From the properties of $H$, and proceeding as in the beginning of the proof of Proposition \ref{contraction}, we get that $W^0\in \mathbb X_{t_0}$ and $TW^0=W^0$. Therefore, by the uniqueness of the fixed point, we obtain that $W^0=Z^0$ in $[-t_0,t_0]$. 
\qed
\section{The dichromatic problem}\label{sec: Dichromatic Lens}

\subsection{Preliminary} \label{prelim}
\subsubsection{Snell's law}\label{subsub:Snell}
Given two homogeneous media $I$ and $II$ separated by a surface $\Gamma$. A monochromatic ray propagating in medium $I$ in the unit direction $\bf x$ is refracted by $\Gamma$ into medium $II$ with unit direction ${\bf m}$ according to Snell's law \cite{BornWolf1999}
\begin{equation}\label{eq:Snell Law}
    n_1({\bf x}\times \boldsymbol \nu)=n_2({\bf m}\times \boldsymbol \nu).
\end{equation}
$\boldsymbol \nu$ is the unit normal to $\Gamma$ at the point of incidence directed toward medium $II$, while $n_1$ and $n_2$ are the refractive indices of media $I$ and $II$ respectively. Moreover, the indices depend on the frequency of the propagating ray \cite{JenkinsWhite2001}.

 From \cite{Gutierrez2013},  \eqref{eq:Snell Law} yields the following formula for the unit direction of the refracted ray
\begin{equation}\label{eq:refracted ray}
{\bf m}=\dfrac{n_1}{n_2} \left({\bf x}-\lambda \boldsymbol \nu\right) 
\end{equation}
where
\begin{equation}\label{eq:mu}
\lambda=\dfrac{1-\left(\dfrac{n_2}{n_1}\right)^2}{{\bf x}\cdot \boldsymbol \nu+\sqrt{\left(\dfrac{n_2}{n_1}\right)^2-1+({\bf x}\cdot \boldsymbol \nu)^2}}.
\end{equation}
For $n_1<n_2$, refraction always occurs and $\lambda<0$. When $n_1>n_2$, refraction occurs only when 
\begin{equation}\label{eq:internal reflection}
{\bf x}\cdot {\boldsymbol \nu}\geq \sqrt{1-\left(\dfrac{n_2}{n_1}\right)^2}
\end{equation}
and, for such directions, $\lambda>0$.

From \eqref{eq:refracted ray}, we get that the vectors ${\bf x, m,}$ and ${\bf \boldsymbol \nu}$ belong to the same plane, the plane of incidence. We then recover the scalar Snell's law
\begin{equation}\label{eq:scalarSnell}
    n_1\sin \theta_1=n_2\sin \theta_2,
\end{equation}
where $\theta_1$ is the angle of incidence between ${\bf x}$ and ${\boldsymbol \nu}$, and $\theta_2$ is the angle of refraction between ${\bf m}$ and ${\boldsymbol \nu}.$

\subsubsection{The Monochromatic problem} \label{prelim:monochromatic}
Let $\Omega$  be a compact subset of the upper unit sphere $S^2$ given by $\Omega={\bf x}(D)$, where $D$ is a convex and compact domain in $\mathbb R^2$. Suppose monochromatic rays are emitted from the origin with unit directions ${\bf x}(t)$  for $t\in D$. Given a surface $L$, the monochromatic problem is concerned with constructing a surface $S$
 so that the lens with lower face $L$ and upper face $S$ refracts all the incident rays into a collimated beam with a given unit direction ${\bf w}$.
 
 For the purpose of this paper, the lens is surrounded by vacuum and the medium inside has refractive index $n>1$. 
Assume that the lower face $L$ is parametrized by $\rho(t){\bf x}(t)$ with $\rho$ a given positive $C^2(D)$ function. For every $t\in D$, the incident ray with direction ${\bf x}(t)$ is refracted by $L$ into the direction ${\bf m}(t)$, given by \eqref{eq:refracted ray}.
  From \cite{Gutierrez2013}, if \begin{equation}\label{eq:cond}
 {\bf m}(t)\cdot {\bf w}\geq \dfrac{1}{n}
 \end{equation}
 for every $t\in D$ then there exists a surface $S$, such that the lens $(L,S)$ refracts all emitted rays into the direction ${\bf w}$. $S$ is parametrized by ${\bf f}(t)=\rho(t){\bf x}(t)+d(t){\bf m}(t)$ with 
 \begin{equation}\label{eq:d}
 d(t)=\dfrac{C-\rho(t)(1-{\bf w}\cdot {\bf x}(t))}{n-{\bf w}\cdot {\bf m}(t)},
 \end{equation}
 where $C$ is a constant chosen so that $ d(t)>0$. Notice that different $C$'s define distinct upper faces, obtaining for each given lower face a family of lenses that solve the monochromatic problem.

\subsection{Setup of the dichromatic problem.} \label{sec:Setup}

In this paper, we are interested in the case when the light emitted from the origin is dichromatic, that is, composed of two frequencies. The goal is to construct a lens (both lower and upper faces are unknown), so that rays of both colors emerge from the lens in the same chosen unit direction $\textbf{w}$, see Figure \ref{fig:Scaled2D}. 
\begin{figure}[ht]
     \centering
     \includegraphics[width=0.5\linewidth]{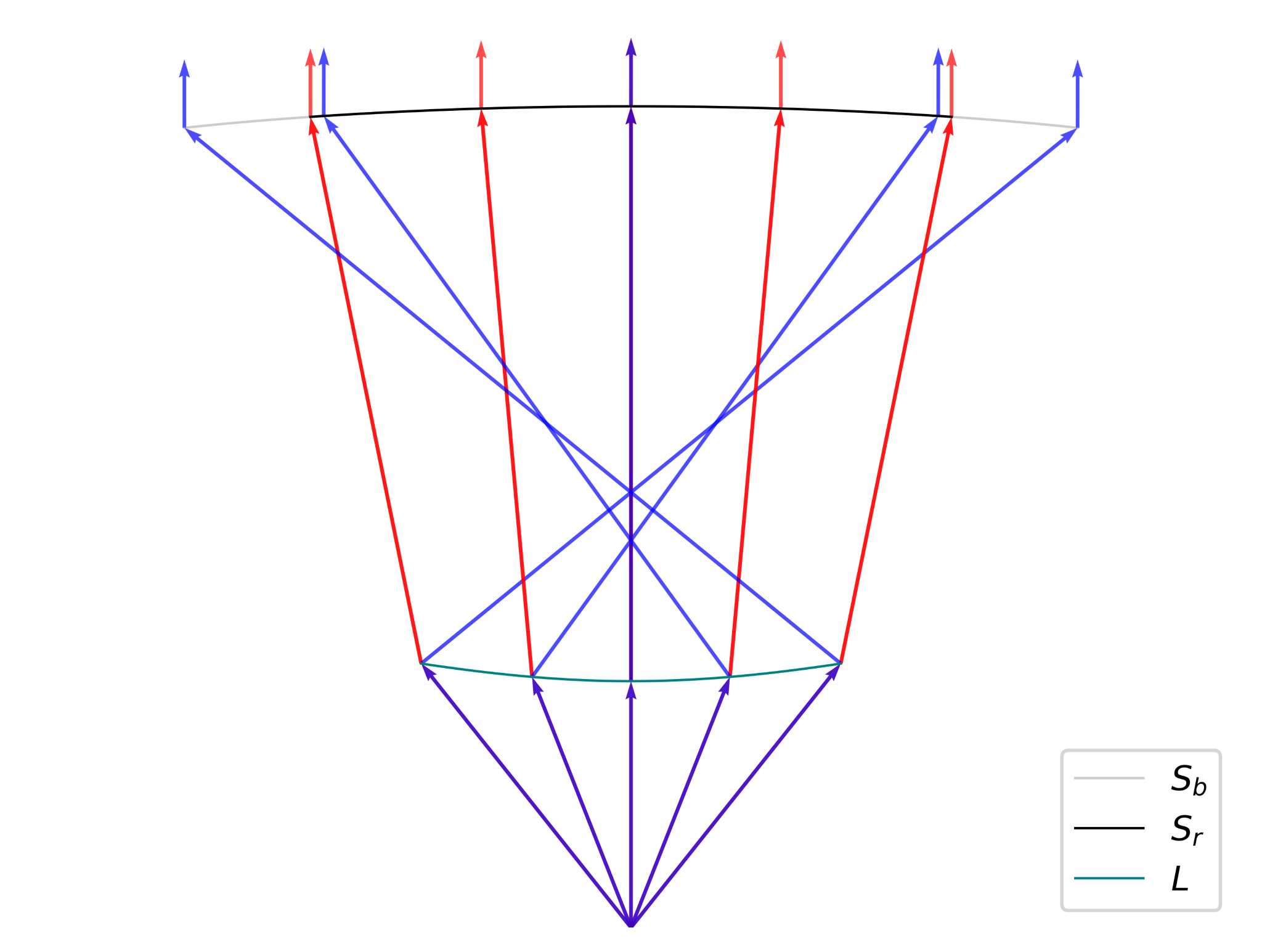}
     \caption{Dichromatic Lens}
     \label{fig:Scaled2D}
 \end{figure}

\noindent Note that waves of different frequencies
composing the incident ray refract at different angles, resulting in the dispersion of light inside the lens.

More precisely, given  $\Omega$ and $D$ as in Section \ref{prelim:monochromatic}, suppose that dichromatic rays composed of two  frequencies $\omega_r$ and $\omega_b$ are emitted from the origin $O$ with unit directions ${\bf x}(t)$, $t\in D$. We define the lenses 
$(L,S_r)$ and $(L,S_b)$, surrounded by vacuum and composed of a homogeneous material with refractive indices $n_r$ and $n_b$ corresponding to the frequencies $\omega_r$ and $\omega_b$. 
$L$ is parametrized by the vector $\rho(t)
{\bf x}(t)$, $t\in D$, with 
$\rho\in C^2(D)$ positive. $S_{r}$ (respectively $S_b$)
is parameterized by ${\bf f}_r(t)=\rho(t){\bf x}
(t)+d_r(t){\bf m}_r(t)$ (${\bf f}_b(t)=\rho(t){\bf x}
(t)+d_b(t){\bf m}_b(t)$). Here, ${\bf 
m}_r$ (${\bf m}_b$) is given by \eqref{eq:refracted 
ray} with $\boldsymbol \nu=\boldsymbol{\nu}_{L}$ the normal to $L$ at $\rho(t){\bf x}(t)$, $n_1=1$ and $n_2=n_r$ $(n_2=n_b)$,  
and $d_r$ ($d_b$) is given by \eqref{eq:d} with 
$n=n_r$ $(n=n_b)$ and $C=C_r$ $(C=C_b)$ chosen such that 
$d_r>0$ $(d_b>0)$. We assume that $n_b>n_r>1$. From Section \ref{prelim:monochromatic}, if the domain $D$ and the direction ${\bf w}$ are  such that ${\bf m}_r(t)\cdot {\bf w}\geq \dfrac{1}{n_r}$ for every $t\in D$, then the lens $(L,S_r)$ refracts the rays with frequency $\omega_r$ into ${\bf w}$.  Similarly, if ${\bf m}_b(t)\cdot {\bf w}\geq \dfrac{1}{n_b}$ then the lens $(L,S_b)$ refracts the rays with frequency $\omega_b$ into ${\bf w}$.

 Our objective is to determine whether there exists a lens $(L,S)$ that solves the monochromatic problem for both frequencies $\omega_r$ and $\omega_b$ simultaneously, which we refer to as the dichromatic problem. In other words, we study the existence of a positive function $\rho\in C^2(D)$, real numbers $C_r$ and $C_b$, a $C^1$ map $\varphi:D\mapsto D$ such that ${\bf m}_r$ and ${\bf m}_b$ satisfy \eqref{eq:cond}, $d_r$ and $d_b$ are positive, ${\bf f}_r$ and ${\bf f}_b$ have a normal at each point, and 
\begin{equation}\label{phi}
{\bf f}_r(t)={\bf f}_b(\varphi(t))  \qquad t\in D,
\end{equation}
see Figure \ref{fig: diagram}.
\begin{figure}[ht] 
\centering
\resizebox{0.4\textwidth}{!}{%
\begin{circuitikz}


  \draw [
    line width=0.5pt,
    dashed, draw opacity = 0.75
  ] (12.5,2.875) -- (12.5,17.8);

\node[black, font=\fontsize{5pt}{5pt}\selectfont] at (12.35, 5.48) {$\varphi(t)$}; 
\draw[color = Blue, line width = 1pt](12.5,4.125)--(11.9, 6.55);
\draw[color=Blue, line width = 1pt] (11.9, 6.55) -- (12.6,15.05);

  \draw [
    color=Crimson,
    draw opacity=1,
    line width=1pt
  ] (12.5,4.125) -- (14,6.7);

  \draw [
    color=Crimson,
    draw opacity=1,
    line width=1pt
  ] (14,6.7)--(12.6,15.05);

\draw[color=black, draw opacity=0.5,-{Stealth[scale=0.7pt]}] (11.9, 6.55)--(12.07,7.6);

\draw[color=black, draw opacity=0.5,-{Stealth[scale=0.7pt]}](14,6.7)--(13.7, 7.6);

\draw[black] (12.5,5.7) arc (90:139:0.45);


  
  \draw [
    color=Crimson,
    draw opacity=1,
    line width=1pt] (12.6,15.05) -- (12.6,17.25);

\draw[color = Blue, line width = 1pt, dashed](12.6,15.05) -- (12.6,17.25);

 \draw[color=red, -{Stealth[scale=1.4, newpurple]}] (12.6,15.05)--(12.6,16.25) node[
      pos=1, right,
      inner xsep=0.080cm, inner ysep=0.085cm,
      rounded corners=0.020cm, font =\small
    ]{\color{newpurple}\textbf{w}};

  \draw [line width=1pt, short]
    (8.93,14.725) .. controls (11.465,15.175) and (13.64,15.175) .. (16.13,14.725)
    node[pos=1, right,]{S};

  \draw [line width=1pt, short]
    (8.522,7.915) .. controls (11.003,6.07) and (13.997,6.07) .. (16.478,7.915)
    node[ pos=1, right]{L};

\draw[black] (12.5,5.5) arc (90:45:0.9);
\node[black, font=\tiny] at (12.7, 5) {$t$};

  \draw [ color=teal,
    line width=1pt,
    {Stealth[scale=0.75]}-{Stealth[scale=0.75]}
  ] (14.416,6.8) -- (13.0,14.95)
    node[
      pos=0.5,
      fill=white, fill opacity=1,
      text opacity=1,
      inner xsep=0.080cm, inner ysep=0.085cm,
      rounded corners=0.020cm,
      text opacity=1, font = \tiny
    ]{{\color{teal}$d_r(t)$}};

  \draw [ color=teal,
    line width=1pt,
    {Stealth[scale=0.75]}-{Stealth[scale=0.75]}
  ] (11.421, 6.6) -- (12.2,14.95)
    node[
      pos=0.5,
      fill=white, fill opacity=1,
      text opacity=1,
      inner xsep=0.01cm, inner ysep=0.085cm,
      rounded corners=0.020cm,
      text opacity=1, font =\tiny
    ]{{\color{teal}$d_b(\varphi(t))$}};
 
  \draw [color=Crimson, line width=1pt,
    -{Stealth[scale=0.75pt]}] 
    (14,6.7) -- (13.826,7.7)
    node[xshift=3.1mm, yshift=1mm, fill=white, fill opacity=1, text opacity=1, inner xsep=0mm, inner ysep=0mm, rounded corners=0.02cm, font =\tiny]{$\textbf{m}_r(t)$};

  \draw [color=Blue, line width=1pt,-{Stealth[scale=0.75pt]}]
   (11.9, 6.55) -- (12,7.7)
    node[xshift=-5.8mm,yshift=0.8mm, fill=white, fill opacity=1, text opacity=1, inner xsep=0.05cm, inner ysep=0.05cm, rounded corners=0.02cm, font =\tiny]{$\textbf{m}_b(\varphi(t))$};

  \draw [
    color=teal,
    draw opacity=1,
    line width=1pt,
    {Stealth[scale=0.75]}-{Stealth[scale=0.75]}
  ] (12.877,4) -- (14.5,6.85)
    node[ xshift = -4mm, yshift=-15mm,
      fill=white, fill opacity=1,
      text opacity=1,
      inner xsep=0.080cm, inner ysep=0.085cm,
      rounded corners=0.020cm,
      text opacity=1, font=\tiny
    ]{{\color{teal}$\rho(t)$}};

\draw[color=teal,
    draw opacity=1,
    line width=1pt,
    {Stealth[scale=0.75]}-{Stealth[scale=0.75]}
  ] (12.031,4)--(11.375,6.65)
    node[ xshift=-2.5mm, yshift=-12mm,
      fill=white, fill opacity=1,
      text opacity=1,
      inner xsep=0.080cm, inner ysep=0.085cm,
      rounded corners=0.020cm,
      text opacity=1, font=\tiny 
    ]{{\color{teal}$\rho(\varphi(t))$}};


    \draw[ color=Crimson, line width = 1pt, -{Stealth[scale=0.75pt]}] (12.5,4.125) -- (13.14,5.2) node[ xshift =3.5mm, yshift=0.4mm , fill=white, fill opacity=1, text opacity=1, inner xsep=0.05cm, inner ysep=0.05cm, rounded corners=0.02cm, font=\tiny] {$\textbf{x}(t)$};

\draw[color = Blue, line width = 1pt, -{Stealth[scale=0.75pt]}](12.5,4.125)--(12.2, 5.3) node[xshift=-5mm, yshift=-1.5mm,fill=white, fill opacity=1, text opacity=1, inner xsep=0.05cm, inner ysep=0.05cm, rounded corners=0.02cm, font=\tiny] {$\textbf{x}(\varphi(t))$};

 \node [circ, color=teal,label={[label distance=1mm]30:\color{teal}$f_r(t) = f_b(\varphi(t))$}] at (12.6,15.05){};

\end{circuitikz}
}%
\caption{Diagram of the dichromatic problem with the corresponding functions}\label{fig: diagram}
\end{figure}
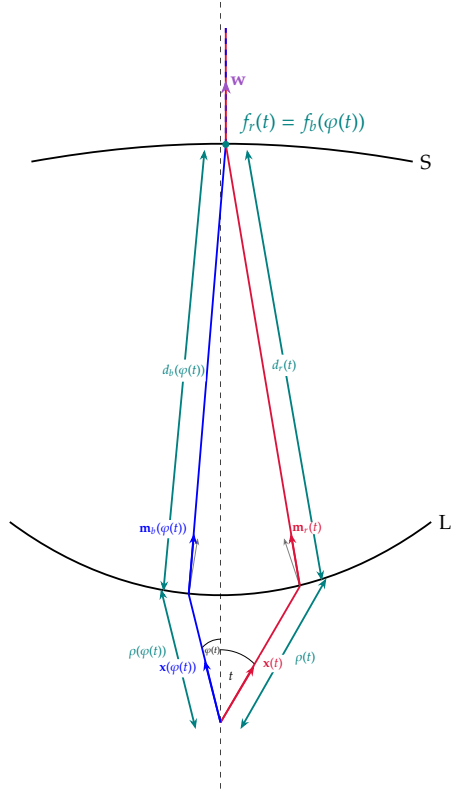

Notice that if a solution exists, then ${\bf f}_r(D)\subseteq {\bf f}_b(D)$, that is, $S_r\subseteq S_b$ and the lens sandwiched between $L$ and $S_b$ refracts rays of both frequencies into ${\bf w}$. In other words, there could be points on ${\bf f}_b(D)$ that are not reached by the rays with frequency $\omega_r$, see Figure \ref{fig:Scaled2D}. 

Now, we provide a necessary condition for the solvability of the dichromatic problem.

\begin{proposition}\label{prop:straight ray}
If the dichromatic problem is solvable, then there exists $s_0\in D$ such that ${\bf x}(s_0)={\bf w}$, that is, ${\bf w}\in {\bf x}(D)$. In addition, we obtain that $$d_r(s_0)=d_b(s_0) \qquad and
\qquad {\bf x}(s_0)={\bf m}_r(s_0)={\bf m}_b(s_0)=\boldsymbol{\nu}_{L} (s_0)={\bf w},$$
where $\boldsymbol \nu_L(s_0)$ is the unit normal to $L$ at $\rho(s_0){\bf x}(s_0)$ toward the medium of the lens.
\end{proposition}

\begin{proof}
     Suppose $\varphi:D\mapsto D$ exists. Then, by the Brouwer Fixed Point Theorem, there exists $s_0\in D$ such that $\varphi(s_0)=s_0$. This implies, from \eqref{phi}, that ${\bf f}_r(s_0)={\bf f}_b(s_0)$ and so $d_r(s_0){\bf m}_r(s_0)=d_b(s_0){\bf m}_b(s_0)$. Since $d_r,d_b>0$, and $|{\bf m}_r|=|{\bf m}_b|=1$, we conclude that $d_r(s_0)=d_b(s_0)$ and ${\bf m}_r(s_0)={\bf m}_b(s_0)$.

     From Section \ref{subsub:Snell}, for every $t\in D$, the vectors ${\bf m}_r(t)$, ${\bf m}_b(t)$, ${\bf x}(t)$, and $\boldsymbol{\nu}_{L}(t)$ belong to the same plane of incidence. Let $\theta(t)$ be the angle of incidence between ${\bf x}(t)$ and $\boldsymbol{\nu}_{L}(t)$, and let $\theta_r(t)$ (respectively $\theta_b(t)$) be the angle of refraction between ${\bf m}_r(t)$ (${\bf m}_b(t)$) and $\boldsymbol{\nu}_{L}(t)$. Since ${\bf m}_r(s_0)={\bf m}_b(s_0)$, $\theta_r(s_0)=\theta_b(s_0)$. From \eqref{eq:scalarSnell} and having that $n_b\neq n_r$, we obtain that $\theta(s_0)=\theta_r(s_0)=\theta_b(s_0)=0$ which yields ${\bf x}(s_0)={\bf m}_b(s_0)={\bf m}_r(s_0)$. With the same argument at the point ${\bf f}_r(s_0)={\bf f}_b(s_0)$, we get that ${\bf m}_r(s_0)={\bf m}_b(s_0)={\bf w}.$
\end{proof}

\subsection{The two dimensional setup}
{}
We consider the dichromatic problem in $\mathbb R^2$. By rotating the plane, we may assume that ${\bf w}=(0,1)$. Suppose rays are emitted from the origin with unit directions ${\bf x}(t)=(\sin t,\cos t)$, for $t$ in a closed interval $D\subseteq (-\pi/2,\pi/2)$ in $\mathbb R$. 
Assume that there exists a solution to the dichromatic problem.
From Proposition \ref{prop:straight ray}, $s_0=0$ belongs to $D$ and is a fixed point of $\varphi.$ As a consequence, 
\begin{equation}\label{eq:phi at 0}
\varphi(0)=0,
\end{equation}
and
\begin{equation}\label{eq:ray at 0}
  {\bf x}(0)={\bf m}_r(0)={\bf m}_b(0)=\boldsymbol{\nu}_{L}(0)={\bf w}=(0,1).  
\end{equation}
Hence, replacing in \eqref{eq:d} and using Proposition \ref{prop:straight ray}, we get
\begin{equation}\label{ref:relation C}
d_r(0)=d_b(0)=\dfrac{C_r}{n_r-1}=\dfrac{C_b}{n_b-1}.
\end{equation}

Moreover, from \cite[Lemma 5.3]{SIAMGUTSAB}, the unit normal $\boldsymbol{\nu}_{L}(t)$ to $L$ at $\rho(t){\bf x}(t)$ is given by 
\begin{equation}\label{eq:formula for normal}
\boldsymbol{\nu}_{L}(t)=\dfrac{1}{\sqrt{\rho(t)^2+\rho'(t)^2}}\left(\rho(t)\sin t-\rho'(t)\cos t, \rho'(t)\sin t+\rho(t)\cos (t)\right),
\end{equation}
and so at $t=0$, \eqref{eq:ray at 0} gives that
\begin{equation}\label{eq: rho' at zero}
    \rho'(0)=0.
\end{equation}

\subsection{Derivation of the System of FDEs} \label{DerivOFsys}

Given $D=[-t_0,t_0]\subseteq (-\pi/2,\pi/2)$, suppose that we have a lens $(L,S)$ that solves the dichromatic problem in two dimensions.
We prove the following Theorem.

\begin{theorem}\label{eq:dichromatic theorem}
  Define $Z_1(t)=\varphi(t)$, $Z_2(t)=\rho(t)-\rho(0)$, and $Z_3(t)=\rho'(t).$ Then $Z(t)=(Z_1(t),Z_2(t),Z_3(t))$ solves a system of FDEs of the form \eqref{eq:system}, with $F=(F_1, F_2, F_3)$ given in \eqref{eq:F1}, \eqref{eq:F2}, \eqref{eq:F3} with the following choice of parameters $\rho_0=\rho(0)$, and $d_0=d_b(0)=d_r(0).$
\end{theorem}

\subsubsection{Auxiliary Functions} Here, we define auxiliary functions and prove some properties that will be used in the proof of Theorem \ref{eq:dichromatic theorem} and in later sections.

We are given $n_r>1$, ${\bf w}=(0,1)$, and the parameters $\rho_0$, $d_0>0$. For $t\in \mathbb R$ and $u=(u_1,u_2)\in \mathbb R^2$, we define the following maps
{\small
\begin{equation}\label{eq:auxiliary}
\begin{cases}
         A_r(u) = \dfrac{1-n_r^2}{u_1+\sqrt{(u_1^2+u_2^2)(n_r^2-1)+u_1^2}}\\
         {\bf M}_r(t,u)=(M_{r1},M_{r2}) = \dfrac{1}{n_r}\left[(\sin t,\cos t)-A_r(u)\left(u_1\sin t-u_2\cos t, u_1\cos t+u_2\sin t\right)\right]\\
        D_r(t,u) = \dfrac{(n_r-1)d_0-u_1(1-\cos t)}{n_r - {\bf w}\cdot {\bf M}_r(t,u)}\\
        {\bf F}_r(t,u)=(F_{r1},F_{r2})=u_1(\sin t,\cos t)+D_r(t,u){\bf M}_r(t,u)\\
        \Lambda_r(t,u)=\sqrt{1+\dfrac{1}{n_r^2}-\dfrac{2}{n_r}{\bf w}\cdot {\bf M}_r(t,u)}\\ \end{cases}
\end{equation}
}
\begin{remark}\label{rmk:Aux obs}  
We make the following observations.

First, notice that $A_r$ and ${\bf M}_r$ are analytic for $t\in \mathbb R$, $u_1>0$, $u_2\in \mathbb R$. The same holds for the remaining functions since 
\begin{equation}\label{eq:unit}
    |{\bf M}_r(t,u)|=1.
\end{equation}
To prove \eqref{eq:unit}, notice that
\begin{align*}
    {\bf M}_r(t,u)\cdot {\bf M}_r(t,u)&=\dfrac{1}{n_r^2}\left[1-2A_r(u)u_1+A_r(u)^2(u_1^2+u_2^2)\right]\\
\end{align*}
From \eqref{eq:auxiliary}, 
$
    A_r(u)=\dfrac{u_1-\sqrt{(u_1^2+u_2^2)(n_r^2-1)+u_1^2}}{u_1^2+u_2^2}.
$
Then,
$$
A_r(u)^2(u_1^2+u_2^2)
=\dfrac{2u_1^2+(n_r^2-1)(u_1^2+u_2^2)-2u_1\sqrt{\left(u_1^2+u_2^2\right)(n_r^2-1)+u_1^2}}{u_1^2+u_2^2}=n_r^2-1+2u_1A_r(u),
$$
and \eqref{eq:unit} follows.
    
Second, we calculate the auxiliary functions in \eqref{eq:auxiliary} at $t=0$ and $ u=(\rho_0,0)$ and obtain
{\small
\begin{align}\label{eqs:identity at 0}
A_r(\rho_0,0)=\dfrac{1-n_r}{\rho_0},\quad {\bf M}_r(0,\rho_0,0&)=(0,1), \quad D_r(0,\rho_0,0)=d_0,\\ {\bf F}_r(0,\rho_0,0)=(0,\rho_0+d_0)&, 
\quad \Lambda_r(0,\rho_0,0)=\dfrac{n_r-1}{n_r}.\notag
\end{align}
}
Moreover, the auxiliary functions satisfy the following symmetry properties
\begin{equation}\label{eq:symmetry nontilde}
    \begin{cases}
       A_r(u_1, -u_2)= A_r(u_1,u_2)\\
       {\bf M}_r(-t,u_1,-u_2) = \left(-M_{r1} (t,u_1,u_2), M_{r2}(t,u_1,u_2)\right) \\
       D_r(-t, u_1, -u_2) = D_r(t, u_1, u_2) \\
       {\bf F}_r (-t, u_1, -u_2) = \left(-F_{r1}(t,u_1,u_2) , F_{r2} (t,u_1,u_2)\right)\\
       \Lambda_r (-t, u_1, -u_2) = \Lambda_r (t,u_1, u_2)
 \end{cases}.
\end{equation}

Finally, we observe that $u_1+n_rD_r-{\bf w}\cdot { \bf F}_r$ is a constant independent of $t$ and $u$. In fact,
\begin{align*}
    u_1+n_rD_r(t,u)-{\bf w}\cdot {\bf F}_r(t,u)&=u_1+n_rD_r(t,u)-u_1{\bf w}\cdot (\sin t,\cos t)+D_r(t,u){\bf w}\cdot {\bf M}_r(t,u)\\
    &=u_1(1-\cos  t)+D_r(t,u)(n_r-{\bf w}\cdot {\bf M}_r(t,u)),
\end{align*}
concluding that
\begin{equation}\label{eq:symbolic optical path}
u_1  +n_rD_r(t,u)-{\bf w}\cdot {\bf F}_r(t,u)=(n_r-1)d_0.
\end{equation}
\end{remark}

Next, we introduce additional auxiliary functions with $n_r>1$, $t\in \mathbb R$, ${\bf w}=(0,1)$, $\rho_0,d_0>0$, and $u=(u_1,u_2)$, $v=(v_1,v_2)\in \mathbb R^2$  

{\footnotesize
\begin{equation}
\begin{cases} \label{eq:aux with tilde}
    \widetilde{A}_r(u,v) = \dfrac{[A_r(u)]^2}{n_r^2-1} \left(v_1+\dfrac{\left(u_1v_1+u_2v_2\right)\left(n_r^2-1\right)+u_1v_1}{\sqrt{(u_2^2 + u_1^2)(n_r^2-1)+u_1^2}}\right)\\
   \widetilde{{\bf M}}_r(t,u,v)=(\widetilde{M}_{r1},\widetilde{M}_{r2})= \frac{1}{n_r}\big((\cos{t}, -\sin{t})- \widetilde {A}_r(u,v) (u_1\sin{t} -u_2\cos{t}, u_1 \cos t+u_2\sin {t}) \\ 
     \qquad\qquad\qquad\qquad\qquad \qquad -A_r(u)((u_1-v_2) \cos{t}+(u_2+v_1) \sin{t},(v_2-u_1)\sin{t}+(v_1+u_2)\cos{t})\big) \\
   \widetilde{D}_r(t,u,v) = \dfrac{ (v_1(\cos{t}-1) - u_1 \sin{t})  + D_r(t,u){\bf w}\cdot \widetilde{{\bf M}}_{r}(t,u, v) } {n_r- {\bf w}\cdot {\bf M}_{r} (t,u) } \\
    \widetilde{{\bf F}}_r(t,u,v)=(\widetilde{F}_{r1},\widetilde{F}_{r2})=
   u_1 (\cos{t},-\sin{t})+v_1(\sin{t}, \cos{t})+ \widetilde{{\bf M}}_r(t,u,v)D_r (t,u)+ {\bf M}_r(t,u) \widetilde{D}_r(t,u,v)\\
   \widetilde{\Lambda}_{r}(t,u,v)=-\dfrac{1}{n_r}\dfrac{{\bf w}\cdot \widetilde{{\bf M}}_r(t,u,v)}{\Lambda_r(t,u)} 
\end{cases}
\end{equation}
}

\begin{remark}\label{rmk:Aux Additional obs}
The functions in \eqref{eq:aux with tilde} are analytic for $t\in \mathbb R$, $u_1>0$, and $u_2,v_1,v_2\in \mathbb R$.

For $t=0$, $u=(\rho_0,0),$ and $ v=(0,v_2)$, we have from \eqref{eqs:identity at 0}
\begin{align}\label{eqs:identities for tilde}
\widetilde{A}_r(\rho_0,&0,0,v_2)=0,\quad \widetilde{{\bf M}}_r (0,\rho_0,0,0,v_2)=\left(1-\frac{n_r-1}{n_r}  \frac{v_2}{\rho_0},0\right),\quad  \widetilde{D}_r(0,\rho_0,0,0,v_2)=0,\\ 
 &\widetilde{{\bf F}}_r(0,\rho_0,0,0,v_2)=\left(\rho_0+d_0\left(1-\frac{n_r-1}{n_r} \frac{v_2}{\rho_0}\right),0\right),
\quad \widetilde{\Lambda}_{r}(0,\rho_0,0,0,v_2)=0.\notag
\end{align}
Moreover, from \eqref{eq:symmetry nontilde}, the functions in \eqref{eq:aux with tilde}  satisfy the following symmetry properties
\begin{equation}\label{eq:symmetry tilde}
\begin{cases}
     \widetilde{A}_r(u_1, -u_2,-v_1, v_2)= -\widetilde{A}_r(u_1,u_2,v_1,v_2)\\
       \widetilde{{\bf M}}_r(-t,u_1,-u_2,-v_1, v_2) = \left(\widetilde{M}_{r1} (t,u_1,u_2,v_1,v_2),-\widetilde{M}_{r2}(t,u_1,u_2,v_1,v_2)\right) \\
       \widetilde{D}_r(-t, u_1, -u_2,-v_1, v_2) = -\widetilde{D}_r(t, u_1, u_2,v_1, v_2) \\
       \widetilde{{\bf F}}_r (-t, u_1, -u_2,-v_1, v_2) = \left(\widetilde{F}_{r1}(t,u_1,u_2, v_1, v_2) , -\widetilde{F}_{r2} (t,u_1,u_2,v_1,v_2)\right)\\
      \widetilde{\Lambda}_{r} (-t, u_1, -u_2,-v_1, v_2) = - \widetilde{\Lambda}_{r} (t,u_1, u_2, v_1, v_2)
       \end{cases}.
\end{equation}

Finally, notice that for a differentiable map $u(t)=(u_1(t), u_2(t))$, we have
\begin{align}\label{eq:der aux}
\dfrac{d}{dt}\left[A_r(u(t))\right]&=\widetilde{A}_r(u(t),u'(t)),\qquad \dfrac{d}{dt}\left[{\bf M}_r(t,u(t))\right]=\widetilde{{\bf M}}_r(t,u(t),u'(t)),\qquad \dfrac{d}{dt}\left[D_r(t,u(t))\right]=\widetilde{D}_r(t,u(t),u'(t))\\
 &\dfrac{d}{dt} \left[{\bf F}_r(t,u(t))\right]=\widetilde {\bf F}_r(t,u(t),u'(t)),\qquad \dfrac{d}{dt}\left[\Lambda_r(t,u(t))\right]=\widetilde{\Lambda}_{r}(t,u(t),u'(t))\notag.
\end{align}
\end{remark}

We prove the following formula which will be used in Section \ref{sec: converse of theorem}.

\begin{proposition}\label{prop:quasicolinearity}
    Given a differentiable function $u(t)=(u_1(t),u_2(t))$ with $t\in \mathbb R$ and $u_1(t)>0$, we have
$$
\left({\bf M}_r(t,u(t))-\dfrac{1}{n_r}{\bf w}\right)\cdot {\widetilde{\bf F}}_r(t,u(t),u'(t))=\dfrac{1}{n_r}A_r(u(t))(u_2(t)-u_1'(t))u_1(t)
$$
\end{proposition}

\begin{proof}
  From \eqref{eq:unit} and \eqref{eq:der aux},  $${\bf M}_r(t,u(t))\cdot {\widetilde {\bf M}}_r(t,u(t),u'(t))=0.$$
  Hence, from the expressions of ${\widetilde{\bf F}}_r$ in \eqref{eq:aux with tilde} and ${\bf M}_r$ in \eqref{eq:auxiliary},
we obtain that
{\small
\begin{align*}
{\bf M}_r(t,u(t))\cdot {\widetilde{\bf F}}_r(t,u(t),u'(t))&=\left(u_1(t)\cos t+u_1'(t)\sin t,-u_1(t)\sin t+u_1'(t)\cos t\right)\cdot {\bf M}_r(t,u(t))+{\widetilde D}_r(t,u(t),u'(t))\\
&=\dfrac{1}{n_r}\left(u_1'(t)-A_r(u(t))\left(u_1(t)u_1'(t)-u_1(t)u_2(t)\right)\right)+\widetilde{D}_r(t,u(t),u'(t))\\
&=\dfrac{1}{n_r}\left(u_1'(t)+n_r\widetilde{D}_r(t,u(t),u'(t))\right)+\dfrac{1}{n_r}A_r(u(t))(u_2(t)-u_1'(t))u_1(t).
\end{align*}
}
From \eqref{eq:symbolic optical path}, for every $t\in D$,
$$u_1(t)+n_rD_r(t,u(t))-{\bf w}\cdot {\bf F}_r(t,u(t))=(n_r-1)d_0.$$
Differentiating with respect to $t$ and using \eqref{eq:der aux}, we get that
$${\bf w}\cdot {\widetilde{\bf F}}_r(t,u(t),u'(t))=u_1'(t)+n_r\widetilde{D}_r(t,u(t),u'(t)).$$
We conclude that
$${\bf M}_r(t,u(t))\cdot {\widetilde{\bf F}}_r(t,u(t),u'(t))=\dfrac{1}{n_r}{\bf w}\cdot {\widetilde{\bf F}}_r(t,u(t),u'(t))+\dfrac{1}{n_r}A_r(u(t))(u_2(t)-u_1'(t))u_1(t),$$
and the proposition follows.
\end{proof}

For $\omega_b$ with $n_b>1$, with $r$ replaced by $b$, we define auxiliary functions similar to those in \eqref{eq:auxiliary} and \eqref{eq:aux with tilde}.

Now, we define $F=(F_1,F_2, F_3)$ corresponding to the system of FDEs \eqref{eq:system}. 

\begin{definition}\label{def:F_i}
 Given $n_r,n_b>1$, and the parameters $\rho_0,d_0>0$.
 Consider the open domain $U\subseteq \mathbb R^{13}$ 
 $$U=\{(t,\alpha,\beta,\zeta,\xi):t\in \mathbb R;\, \alpha,\beta,\zeta,\xi\in \mathbb R^3;\, \alpha_2,\beta_2>-\rho_0\}.$$ We define the map $F=(F_1,F_2,F_3)$ on $U$
as follows
\begin{equation}\label{eq:F1}
F_1(t,\alpha,\beta,\zeta,\xi)=
\widetilde{F}_{r1}(t,\alpha_2+\rho_0,\alpha_3,\zeta_2,\zeta_3)-\zeta_1\widetilde{F}_{b1}(\alpha_1
,\beta_2+\rho_0,\beta_3,\xi_2,\xi_3),\\
\end{equation}
\begin{equation} \label{eq:F2}
    \begin{aligned}
    F_2(t,\alpha,\beta,\zeta,\xi)&= 
     \widetilde{M}_{r1}(t,\alpha_2+\rho_0,\alpha_3,\zeta_2,\zeta_3)\Lambda_b(\alpha_1, \beta_2 + \rho_0,\beta_3)\\
    &\qquad + \zeta_1 M_{r1}(t,\alpha_2 + \rho_0,\alpha_3)\widetilde{\Lambda}_{b}(\alpha_1,\beta_2+\rho_0,\beta_3,\xi_2,\xi_3)\\
    &\qquad\quad-\zeta_1\widetilde{M}_{b1}(\alpha_1,\beta_2+\rho_0,\beta_3,\xi_2,\xi_3)\Lambda_r(t, \alpha_2 + \rho_0,\alpha_3)\\
    &\qquad\qquad\quad- M_{b1}(\alpha_1, \beta_2 + \rho_0,\beta_3)\widetilde{\Lambda}_{r}(t,\alpha_2+\rho_0,\alpha_3,\zeta_2,\zeta_3),
    \end{aligned} 
\end{equation}
and 
\begin{equation}\label{eq:F3}
F_3(t,\alpha,\beta,\zeta,\xi)=\zeta_2-\alpha_3.
\end{equation}
From Remarks \ref{rmk:Aux obs} and \ref{rmk:Aux Additional obs}, $F$ is analytic in $U$. 
\end{definition}
We are interested in the system \eqref{eq:system} with $F_1,F_2,$ and $F_3$ given above. More precisely, we are looking for a $C^1$ function $Z(t)=(Z_1(t),Z_2(t),Z_3(t))$ in a neighborhood of $t=0$ such that $Z(0)={\bf 0}$ and
{\small
\begin{equation}\label{eq:F1explicit}
\begin{aligned}
\widetilde{F}_{r1}&(t,Z_2(t)+\rho_0,Z_3(t),Z_2'(t),Z_3'(t))\\
&\qquad-Z_1'(t)\widetilde{F}_{b1}(Z_1(t),Z_2(Z_1(t))+\rho_0,Z_3(Z_1(t)),Z_2'(Z_1(t)),Z_3'(Z_1(t)))=0
\end{aligned}
\end{equation}

\begin{equation}\label{eq:F2explicit}
\begin{aligned}
\widetilde{M}_{r1}&(t,Z_2(t)+\rho_0,Z_3(t),Z_2'(t),Z_3'(t))
    \Lambda_b(Z_1(t), Z_2(Z_1(t))+\rho_0, Z_3(Z_1(t)))\\
     &\quad+Z_1'(t)M_{r1}(t,Z_2(t)+\rho_0,Z_3(t))\widetilde{\Lambda}_{b}(Z_1(t),Z_2(Z_1(t))+\rho_0,Z_3(Z_1(t)),Z_2'(Z_1(t)),Z_3'(Z_1(t)))\\
    & \quad\quad- Z_1'(t)\widetilde{M}_{b1}(Z_1(t),Z_2(Z_1(t))+\rho_0,Z_3(Z_1(t)),Z_2'(Z_1(t)),Z_3'(Z_1(t)))\Lambda_r(t,Z_2(t)+\rho_0,Z_3(t)) \\
    &\quad\quad\quad\quad-M_{b1}(Z_1(t), Z_2(Z_1(t))+\rho_0, Z_3(Z_1(t)))\widetilde{\Lambda}_{r}(t,Z_2(t)+\rho_0,Z_3(t),Z_2'(t),Z_3'(t))=0
\end{aligned}
\end{equation}

\begin{equation}\label{eq:F3explicit}
    Z_2'(t)-Z_3(t) = 0
\end{equation}}
\subsubsection{Proof of Theorem \ref{eq:dichromatic theorem}} \label{optics_aux}
We assume that there exist $\rho$, $C_r$, $C_b$, and $\varphi$ such that the corresponding lens $(L,S)$ is a solution to the dichromatic problem as described in Section \ref{sec:Setup}. We let $\rho_0=\rho(0)$, $d_0=d_r(0)=d_b(0)$. 

We express the functions related to the optical problem in terms of the auxiliary functions in \eqref{eq:auxiliary}. Define $u(t)=(\rho(t),\rho'(t))$. From \eqref{eq: rho' at zero},  $u(0)=(\rho_0,0)$.
From \eqref{eq:refracted ray}, the incident ray with frequency $\omega_r$ and direction ${\bf x}(t)=(\sin t,\cos t)$ refracts inside the lens with direction
$ {\bf m}_r(t)=\dfrac{1}{n_r}\left({\bf x}(t)-\lambda_{L,r}(t) \boldsymbol{\nu}_{L}(t)\right).$
From \eqref{eq:formula for normal},
\[{\bf x}(t)\cdot \boldsymbol{\nu}_{L}(t)= \dfrac{\rho(t)}{\sqrt{\rho(t)^2+\rho'(t)^2}},\]and so from \eqref{eq:mu}
\begin{align*}
\lambda_{L,r}(t)=\dfrac{(1-n_r^2)\sqrt{\rho(t)^2+\rho'(t)^2}}{\rho(t)+\sqrt{(n_r^2-1)\left(\rho(t)^2+\rho'(t)^2\right)+\rho(t)^2}}.
\end{align*}
Replacing the expressions of $\lambda_{L,r}$ and $\boldsymbol{\nu}_{L}$ in ${\bf m}_r$,
we get that
\begin{equation}\label{eq:mrM r}
\text{\footnotesize ${\bf m}_r(t)=\dfrac{1}{n_r}\left[(\sin t,\cos t)-\dfrac{(1-n_r^2)\left(\rho(t)\sin t-\rho'(t)\cos t, \rho'(t)\sin t+\rho(t)\cos (t)\right)}{\rho(t)+\sqrt{(n_r^2-1)\left(\rho(t)^2+\rho'(t)^2\right)+\rho(t)^2}}\right]={\bf M}_r(t,u(t)).$}
\end{equation}

From \eqref{eq:d} and \eqref{ref:relation C},
\begin{equation}\label{eq:drDr}
d_r(t)=\dfrac{C_r-\rho(t)(1-{\bf w}\cdot (\sin t,\cos t))}{n_r-{\bf w}\cdot {\bf m}_r(t)}=D_r(t,u(t)).   
\end{equation}
Hence, 
\begin{equation}\label{eq:frFr}
{\bf f}_r(t)=\rho(t){\bf x}(t)+d_r(t){\bf m}_r(t)={\bf F}_r(t,u(t)).
\end{equation}

At the point ${\bf f}_r(t)$, the ray with direction ${\bf m}_r(t)$ is refracted into the direction ${\bf w}=(0,1)$. Then, by Snell's law \eqref{eq:refracted ray},
$${\bf w}-\dfrac{1}{n_r}{\bf m_r}(t)=\lambda_{S,r}(t)\boldsymbol \nu_{S_r}(t),$$ with $\lambda_{S,r}>0$ and $\boldsymbol \nu_{S_r}(t)$ the unit normal to $S_r$ at ${\bf f}_r(t)$. Therefore,

\begin{equation}\label{eq: mu 2r}
  \lambda_{S,r}(t)=\left|{\bf w}-\dfrac{1}{n_r}{\bf m_r}(t)\right|=\sqrt{1+\dfrac{1}{n_r^2}-\dfrac{2}{n_r}{\bf w\cdot {\bf m_r}(t)}}=\Lambda_r(t,u(t)).
\end{equation}
Similarly for the rays with frequency $\omega_b$, ${\bf m}_b(t)={\bf M}_b(t,u(t))$, $d_b(t)=D_b(t,u(t))$, ${\bf f}_b(t)={\bf F}_b(t,u(t))$, and $\lambda_{S,b}(t)=\Lambda_b(t,u(t)).$

We have $Z_1(t)=\varphi(t)$, $Z_2(t)=\rho(t)-\rho(0)$, and $Z_3(t)=\rho'(t)$. From \eqref{eq:phi at 0} and \eqref{eq: rho' at zero}, $Z(0)={\bf 0}$.
Additionally, notice that $u(t)=(\rho(t),\rho'(t))=(Z_2(t)+\rho_0,Z_3(t))$ and $u'(t)=(Z_2'(t),Z_3'(t))$.

Differentiating \eqref{phi} with respect to $t$, we get
${\bf f}_r'(t)=\varphi'(t){\bf f}_b'(\varphi(t)).$ Then, from \eqref{eq:der aux} and \eqref{eq:frFr},
$$\widetilde{F}_{r1}(t,u(t),u'(t))=Z_1'(t)\widetilde{F}_{b1}(Z_1(t),u(Z_1(t)),u'(Z_1(t))).$$
Hence, \eqref{eq:F1explicit} follows.

At the point ${\bf f}_r(t)={\bf f}_b(\varphi(t))$, the rays ${\bf m}_r(t)$ and ${\bf m}_b(\varphi(t))$  are refracted by S  into ${\bf w}=(0,1)$, see Figure \ref{fig: diagram}. Then, by Snell's law \eqref{eq:refracted ray},
$$
    {\bf m}_r(t) - \dfrac{1}{n_r}{\bf w} =\lambda_{S,r}(t)\boldsymbol \nu_{S_r}(t) \qquad and\qquad
    {\bf m}_b(\varphi(t)) - \dfrac{1}{n_b}{\bf w} = \lambda_{S,b}(\varphi(t))\boldsymbol \nu_{S_b}(\varphi(t)).
$$
Since $\boldsymbol{\nu}_{S_r}(t)=\boldsymbol{\nu}_{S_b}(\varphi(t))$, by comparing the first components of the above equations, we get
\begin{equation*}
    m_{r1}(t)\lambda_{S,b}(\varphi(t)) = m_{b1}(\varphi(t))\lambda_{S,r}(t).
\end{equation*}
Differentiating both sides and using  \eqref{eq:der aux}, \eqref{eq:mrM r} and \eqref{eq: mu 2r}, we obtain

\begin{align*}
\widetilde{M}_{r1}&(t,u(t),u'(t))
    \Lambda_b(Z_1(t), u(Z_1(t))) +Z_1'(t)M_{r1}(t,u(t))\widetilde{\Lambda}_{b}(Z_1(t),u(Z_1(t)),u'(Z_1(t)))\\
    &\quad = Z_1'(t)\widetilde{M}_{b1}(Z_1(t),u(Z_1(t)),u'(Z_1(t)))\Lambda_r(t,u(t))+M_{b1}(Z_1(t),u(Z_1(t)))\widetilde{\Lambda}_{r}(t,u(t),u'(t)).
\end{align*}
Hence, \eqref{eq:F2explicit} follows.

Finally, \eqref{eq:F3explicit} follows from the fact that   $Z_2'(t)=\rho'(t)=Z_3(t)$. 
\qed

\subsection{Converse of Theorem \ref{eq:dichromatic theorem}}\label{sec: converse of theorem}
In this section, we show that from a solution (if it exists) to \eqref{eq:system},  with $F$ as 
in Definition \ref{def:F_i}, we can construct a lens that solves the dichromatic problem.
\begin{theorem}\label{thm:converse}
Let $\rho_0, d_0>0$, $n_b>n_r>1$, ${\bf w}=(0,1)$, and $F=(F_1,F_2,F_3)$, with $F_i$ defined in \eqref{eq:F1}, \eqref{eq:F2} and \eqref{eq:F3}. Suppose there exists ${\bf p}=(p_1,p_2,p_3)$ such that $F(0,{\bf 0}, {\bf 0}, {\bf p}, {\bf p})={\bf 0}$ and 
\begin{equation}\label{eq:p1}
    0<|p_1|<1.
\end{equation}
Assume the system
\begin{equation*}
    \begin{cases}
    F(t,Z(t),Z(Z_1(t)),Z'(t),Z'(Z_1(t)))={\bf 0} \\
    Z(0)={\bf 0}
    \end{cases}
\end{equation*}
has a $C^1$ solution in a neighborhood of $t=0$ with $Z'(0)={\bf p}$.

Let $u(t)=(Z_2(t)+\rho_0,Z_3(t))$.
Then, there exists $t_0>0$ such that for $|t|\leq t_0$
\begin{enumerate}[label=\roman*)]
    \item $|Z_1(t)|\leq |t|$,
    \item $Z_2(t)+\rho_0>0$,
    \item ${\bf M}_r(t,u(t))\cdot {\bf w}\geq \dfrac{1}{n_r}$ and $ {\bf M}_b(t,u(t)) \cdot {\bf w}\geq \dfrac{1}{n_b}$,
    \item $D_r(t,u(t)),D_b(t,u(t))>0$,
    \item $\widetilde {\bf F}_r(t,u(t),u'(t)),\widetilde {\bf F}_b(t,u(t),u'(t))\neq {\bf 0}$,
    \item ${\bf F}_r(t,u(t))={\bf F}_b(Z_1(t),u(Z_1(t))).$
\end{enumerate}



\end{theorem}

\begin{proof}
Inequality $i)$ follows in a neighborhood of $t=0$ from the fact that $Z_1(0)=0$, $Z_1'(0)=p_1$, and \eqref{eq:p1}. 
Similarly,  $\rho_0>0$ and $Z_2(0)$ give inequality $ii)$. 

We have $u(0)=(\rho_0,0)$, so from \eqref{eqs:identity at 0}
$$ {\bf M}_r(0,u(0)) \cdot {\bf w} = {\bf M}_b(0,u(0)) \cdot {\bf w}=1> \dfrac{1}{n_r}>\dfrac{1}{n_b},$$
which implies $iii)$. Also from \eqref{eqs:identity at 0}, $D_r(0,u(0))=D_b(0,u(0))=d_0>0$, obtaining $iv)$.

Next, we show $v)$. From \eqref{eq:F1explicit} at $t=0$,
$$
{\widetilde F}_{r1}(0,u(0),u'(0))=Z_1'(0) {\widetilde F}_{b1}(0,u(0),u'(0))
.$$
From \eqref{eq:p1}, $Z_1'(0)=p_1\neq 0$. Hence, $\widetilde F_{r1}(0,u(0),u'(0))=0$ if and only if $\widetilde F_{b_1}(0,u(0),u'(0))=0$.

From \eqref{eq:F3explicit} at $t=0$, $Z_2'(0)=Z_3(0)=0$. Then, $u'(0)=(Z_2'(0),Z_3'(0))=(0,p_3)$.
For the sake of contradiction, assume that ${\widetilde F}_{r1}(0,u(0),u'(0))={\widetilde F}_{b_1}(0,u(0),u'(0))=0$ which, using \eqref{eqs:identities for tilde}, implies that
$$
 \rho_0+d_0\left(1-\dfrac{n_r-1}{n_r}\dfrac{p_3}{\rho_0}\right)=\rho_0+d_0\left(1-\dfrac{n_b-1}{n_b}\dfrac{p_3}{\rho_0}\right)=0.
 $$
Since $n_r\neq n_b$ and $d_0>0$, the left equality implies that $p_3=0$ which,   substituted in the right equality, yields the contradiction   $\rho_0+d_0=0$. We then obtain $v)$.

It remains to show $vi)$. From \eqref{eq:F1explicit},
\begin{equation}\label{eq:Colinearity 1}
    \widetilde{F}_{r1}(t,u(t),u'(t))=Z_1'(t)\widetilde{F}_{b1}(Z_1(t),u(Z_1(t)),u'(Z_1(t))).
\end{equation}
We prove a similar identity for $\widetilde{F}_{r2}$ and $\widetilde{F}_{b2}$.

From \eqref{eq:F3explicit}, $u_1'(t)=Z_2'(t)=Z_3(t)=u_2(t)$. Therefore, from Proposition \ref{prop:quasicolinearity} 
\begin{equation}\label{eq:orthogonality}
    \begin{cases}
        \left({\bf M}_r(t,u(t))-\dfrac{1}{n_r}{\bf w}\right)\cdot{\widetilde{\bf F}}_r(t,u(t),u'(t))=0 \\
        \left({\bf M}_b(Z_1(t),u(Z_1(t)))-\dfrac{1}{n_b}{\bf w}\right)\cdot{\widetilde{\bf F}}_b(Z_1(t),u(Z_1(t)),u'(Z_1(t)))=0
    \end{cases}.
\end{equation}

\paragraph{\bf Claim.} The vectors ${\bf M}_r(t,u(t))-\dfrac{1}{n_r}{\bf w}$ and ${\bf M}_b(Z_1(t),u(Z_1(t)))-\dfrac{1}{n_b}{\bf w}$ are collinear. 

\begin{proof}[Proof of the claim]Integrating \eqref{eq:F2explicit} and using \eqref{eq:der aux}, we get that\\  
    $M_{r1}(t,u(t)) \Lambda_b(Z_1(t),u(Z_1(t)))=M_{b1}(Z_1(t),u(Z_1(t)))\Lambda_r(t,u(t))+c$.
Plugging $t=0$,  \eqref{eqs:identity at 0} implies that $c=0$ and so 
\begin{equation} \label{eq:M1_Lambda_ratio}
    M_{r1}(t,u(t)) \Lambda_b(Z_1(t),u(Z_1(t)))=M_{b1}(Z_1(t),u(Z_1(t)))\Lambda_r(t,u(t)).
\end{equation}

From \eqref{eq:auxiliary} and \eqref{eq:unit},
\begin{align*}
\Lambda_r(t,u(t))= \sqrt{1+\dfrac{1}{n_r^2}-\dfrac{2}{n_r} M_{r2}(t,u(t))} &= \sqrt{M_{r1}(t,u(t))^2+\left(M_{r2}(t,u(t))-\dfrac{1}{n_r}\right)^2},\\
\Lambda_b(Z_1(t),u(Z_1(t)))&= \sqrt{M_{b1}(Z_1(t),u(Z_1(t)))^2+\left(M_{b2}(Z_1(t),u(Z_1(t))-\frac{1}{n_b}\right)^2}
\end{align*}

Squaring \eqref{eq:M1_Lambda_ratio}, substituting the above expressions of $\Lambda_r$ and $\Lambda_b$, and simplifying, we get that
{\small
\begin{equation}\label{eq:squares}
\left(M_{r1}(t,u(t))\left(M_{b2}(Z_1(t),u(Z_1(t)))-\dfrac{1}{n_b}\right) \right)^2=\left(M_{b1}(Z_1(t),u(Z_1(t)))\left(M_{r2}(t,u(t))-\dfrac{1}{n_r}\right) \right)^2.
\end{equation}
}
From \eqref{eq:M1_Lambda_ratio} and since $\Lambda_r$,$\Lambda_b>0$, $M_{r1}(t,u(t))$ and $M_{b1}(Z_1(t),u(Z_1(t)))$ have the same sign. Consequently, taking the square root in \eqref{eq:squares} and using $iii)$, we obtain
\begin{equation*}
 M_{r1}(t,u(t))\left(M_{b2}(Z_1(t),u(Z_1(t))) - \frac{1}{n_b}\right) = M_{b1}(Z_1(t),u(Z_1(t)))\left(M_{r2}(t,u(t)) - \frac{1}{n_r}\right),  
\end{equation*}
and the claim follows.
\end{proof}

From the orthogonality in \eqref{eq:orthogonality} and the Claim, we conclude that the vectors $\widetilde{F}_r(t,u(t),u'(t))$ and $\widetilde{F}_b(Z_1(t),u(Z_1(t)),u'(Z_1(t)))$ are colinear. Therefore, from \eqref{eq:Colinearity 1}, we get
\begin{equation}\label{eq:Colinearity 2}
\widetilde{F}_{r2}(t,u(t),u'(t))=Z_1'(t)\widetilde{F}_{b2}(Z_1(t),u(Z_1(t)),u'(Z_1(t))).
\end{equation}
By integrating \eqref{eq:Colinearity 1} and \eqref{eq:Colinearity 2}, and using \eqref{eq:der aux},  
$${\bf F}_r(t,u(t))={\bf F}_b(Z_1(t),u(Z_1(t)))+(c_1,c_2).$$
Plugging $t=0$ and using \eqref{eqs:identity at 0}, we get that $c_1=c_2=0$, concluding $vi)$.
\end{proof}

    
\subsubsection{Interpretation of Theorem \ref{thm:converse}}\label{subsub:converse} With the setting of Theorem \ref{thm:converse}, define $\rho(t)=Z_2(t)+\rho_0$, $t\in [-t_0,t_0]$. From $ii)$, $\rho(t)>0$. Let $L=\{\rho(t){\bf x}(t)\}$ with ${\bf x}(t)=(\sin t,\cos t)$ be the curve separating vacuum and the medium with refractive indices $n_r$ and $n_b$ corresponding to $\omega_r$ and $\omega_b$ respectively. From Section \ref{optics_aux}, the ray emitted from the origin with frequency $\omega_r$ and direction ${\bf x}(t)$ is refracted by $L$ into the direction 
    $${\bf m}_r(t)={\bf M}_r(t,\rho(t),\rho'(t))={\bf M}_r(t,Z_2(t)+\rho_0,Z_3(t)).$$
    Let $C_r=(n_r-1)d_0$, and define
    $$d_r(t)=\dfrac{C_r-\rho(t)(1-{\bf x}(t)\cdot {\bf w})}{n_r-{\bf w}\cdot {\bf m}_r(t)}=D_r(t,Z_2(t)+\rho_0,Z_3(t)).$$ From $iv)$, $d_r(t)>0$. We then define $S_r$ parameterized by
    $${\bf f}_r(t)=\rho(t){\bf x}(t)+d_r(t){\bf m}_r(t)={\bf F}_r(t,Z_2(t)+\rho_0,Z_3(t)).$$
    From \eqref{eq:der aux} and $v)$, 
    ${\bf f}'_r(t)=\widetilde {\bf F}_r(t,Z_2(t)+\rho_0,Z_3(t),Z_2'(t),Z_3'(t))\neq 0$, and  hence $S_r$ has a normal at each point.
    Finally,  from $iii)$, ${\bf m}_r(t)\cdot {\bf w}\geq \dfrac{1}{n_r}$. Then, from Section \ref{prelim:monochromatic}, the lens $(L,S_r)$ refracts all emitted rays with frequency $\omega_r$ and directions ${\bf x}(t)$, for $t\in [-t_0,t_0]$, into the direction ${\bf w}$. Similarly, the lens $(L,S_b)$, with $S_b$ parameterized by ${\bf f}_b(t)={\bf F}_b(t,Z_2(t)+\rho_0,Z_3(t))$, refracts all emitted rays with frequency $\omega_b$ into the direction ${\bf w}$.
    
    Let $\varphi(t)=Z_1(t)$. From $i)$, $\varphi:[-t_0,t_0]\mapsto [-t_0,t_0]$, and from $vi)$
    $${\bf f}_r(t)={\bf f}_b(\varphi(t)).$$
    We conclude that the lens $(L,S)$, with $S=S_b$, solves the dichromatic problem for both frequencies $\omega_r$ and $\omega_b.$




\section{Existence of a Solution to the Dichromatic Problem}\label{sec:existence}

\subsection{The two dimensional case}\label{subsec:2D}
In section \ref{sec: converse of theorem}, we showed the equivalence between the existence of a solution to the dichromatic problem in two dimensions and the existence of a solution to the system of FDEs \eqref{eq:system} with $F$ given in Definition \ref{def:F_i}.

In this section, we use Theorem \ref{thm:Main Theorem} to show that such a solution exists for some values of the parameters $\rho_0$ and $ d_0$. We use the notation $k_0=\dfrac{\rho_0}{d_0}$, $\Delta_r=\dfrac{n_r}{n_r-1}$, and $\Delta_b=\dfrac{n_b}{n_b-1}$.

\begin{theorem}
\label{thm:existenceANDuniquenessOFsolution}
    Given $\rho_0,d_0>0$ and $n_b>n_r>1$. For $k_0=\dfrac{\rho_0}{d_0}<\dfrac{(\Delta_r-\Delta_b)^2}{4\Delta_r\Delta_b}$, the system \eqref{eq:system}, with $F=(F_1,F_2,F_3)$  given in Definition \ref{def:F_i},
    has a  unique solution $Z$ in a neighborhood of $t=0$, with $Z'(0)={\bf p}$ where   
\begin{equation}\label{eq:admissible p}
    {\bf p}=\left(\dfrac{\Delta_b-\Delta_r + \sqrt{(\Delta_r-\Delta_b)^2-4k_0\Delta_r\Delta_b}}{\Delta_r-\Delta_b + \sqrt{(\Delta_r-\Delta_b)^2-4k_0\Delta_r\Delta_b}},0,
\dfrac{\Delta_b+\Delta_r + \sqrt{(\Delta_r-\Delta_b)^2-4k_0\Delta_r\Delta_b}}{2}\rho_0\right).
\end{equation}

Additionally, the solution $Z(t)=(Z_1(t),Z_2(t),Z_3(t))$ satisfies the following symmetry properties
\begin{equation}\label{eq:symmetry of solution}
    Z_1(-t)=-Z_1(t),\qquad\qquad Z_2(-t)=Z_2(t),\qquad\qquad Z_3(-t)=-Z_3(t).
\end{equation}

\end{theorem}

\begin{lemma}\label{lm:system at 0}
   The system
    \begin{equation}\label{eq:system at 0}
        F(0,{\bf0}, {\bf 0}, {\bf p
    },{\bf p})={\bf 0}
    \end{equation}
has a solution ${\bf p}=(p_1,p_2,p_3)$ if and only if $k_0\leq \dfrac{(\Delta_r-\Delta_b)^2}{4\Delta_r\Delta_b}$. 
\end{lemma}
\begin{proof}
From \eqref{eq:F3}, $F_3(0,{\bf 0},{\bf 0}, {\bf p}, {\bf p})=0$ if and only if $p_2=0$.
 Then, using \eqref{eq:F1},\eqref{eq:F2},\eqref{eqs:identity at 0}, and \eqref{eqs:identities for tilde}, $F_1(0,{\bf 0},{\bf 0}, {\bf p}, {\bf p})= F_2(0,{\bf 0},{\bf 0}, {\bf p}, {\bf p})=0$ implies
\begin{align}
      \label{eq:algebraic system 1} \rho_0+d_0\left(1-\frac{1}{\Delta_r}\frac{p_3}{\rho_0}\right)-p_1\left(\rho_0+d_0\left(1-\frac{1}{\Delta_b}\frac{p_3}{\rho_0}\right)\right)=0\\
    \label{eq:algebraic system 2} \frac{1}{\Delta_b}\left(1-\frac{1}{\Delta_r}\frac{p_3}{\rho_0}\right)-\frac{p_1}{\Delta_r}\left(1-\frac{1}{\Delta_b}\frac{p_3}{\rho_0}\right) = 0.
\end{align} 
 Multiplying \eqref{eq:algebraic system 2} by $\rho_0+d_0\left(1-\dfrac{1}{\Delta_b}\dfrac{p_3}{\rho_0}\right)$ and using \eqref{eq:algebraic system 1}, we get
\begin{align*}
0&=\dfrac{1}{\Delta_b}\left(1-\dfrac{1}{\Delta_r}\dfrac{p_3}{\rho_0}\right)\left(\rho_0+d_0\left(1-\frac{1}{\Delta_b}\frac{p_3}{\rho_0}\right)\right)-\dfrac{1}{\Delta_r}\left(1-\dfrac{1}{\Delta_b}\dfrac{p_3}{\rho_0}\right)   \left(\rho_0+d_0\left(1-\frac{1}{\Delta_r}\frac{p_3}{\rho_0}\right)\right)\\
     &=\rho_0\left(\dfrac{1}{\Delta_b}-\dfrac{1}{\Delta_r}\right)+d_0\left(\dfrac{1}{\Delta_b}-\dfrac{1}{\Delta_r}\right)\left(1-\left(\dfrac{1}{\Delta_b}+\dfrac{1}{\Delta_r}\right)\dfrac{p_3}{\rho_0}+\dfrac{1}{\Delta_r\Delta_b}\left(\dfrac{p_3}{\rho_0}\right)^2\right)
\end{align*} 
Dividing by $d_0\left(\dfrac{1}{\Delta_b}-\dfrac{1}{\Delta_r}\right)\dfrac{1}{\Delta_r\Delta_b}$, we obtain the following quadratic equation in $\dfrac{p_3}{\rho_0}$
\begin{equation}\label{eq:quadratic-equation}
\left(\dfrac{p_3}{\rho_0}\right)^2-(\Delta_r+\Delta_b)\left(\dfrac{p_3}{\rho_0}\right)+\Delta_r\Delta_b(k_0+1)=0.
\end{equation}
 The discriminant of \eqref{eq:quadratic-equation}
 is $$\delta=(\Delta_r+\Delta_b)^2-4\Delta_r\Delta_b(k_0+1)=(\Delta_r-\Delta_b)^2-4k_0\Delta_r\Delta_b,$$ which is non-negative if and only if $k_0\leq \dfrac{(\Delta_r-\Delta_b)^2}{4\Delta_r\Delta_b}$. 
 
 For such $k_0$, the solutions are
$$\dfrac{p_3}{\rho_0}=\dfrac{\Delta_r+\Delta_b \pm \sqrt{(\Delta_r-\Delta_b)^2-4k_0\Delta_r\Delta_b}}{2}.$$
Notice that, since $k_0\Delta_r\Delta_b>0$,
\begin{equation}\label{eq:p3notDeltab}
   \dfrac{p_3}{\rho_0}-\Delta_b=\dfrac{\Delta_r-\Delta_b\pm \sqrt{(\Delta_r-\Delta_b)^2-4k_0\Delta_r\Delta_b}}{2}\neq 0.
\end{equation}
 Replacing \eqref{eq:p3notDeltab} in \eqref{eq:algebraic system 2}, we get 
\begin{equation}\label{eq:p1explicit}
    p_1=\dfrac{\dfrac{p_3}{\rho_0}-\Delta_r}{\dfrac{p_3}{\rho_0}-\Delta_b}=\dfrac{\Delta_b-\Delta_r \pm \sqrt{(\Delta_r-\Delta_b)^2-4k_0\Delta_r\Delta_b}}{\Delta_r-\Delta_b \pm \sqrt{(\Delta_r-\Delta_b)^2-4k_0\Delta_r\Delta_b}}.
\end{equation}
\end{proof}
\begin{remark}\label{rmk:admissibility of p}
From Theorems \ref{thm:Main Theorem} and \ref{thm:converse}, we are interested in a solution ${\bf p}=(p_1,p_2,p_3)$ to \eqref{eq:system at 0} satisfying \eqref{eq:p1}. From Lemma \ref{lm:system at 0}, notice the following
\begin{itemize}
\item For $k_0>\dfrac{(\Delta_r-\Delta_b)^2}{4\Delta_r\Delta_b}$, \eqref{eq:system at 0} has no real solution, and so, from Remark \ref{rmk:necessary condition}, the system \eqref{eq:system} has no solution.
\item For $k_0=\dfrac{(\Delta_r-\Delta_b)^2}{4\Delta_r\Delta_b}$, the solution to \eqref{eq:system at 0} is ${\bf p}=\left(-1,0,\dfrac{\Delta_b+\Delta_r}{2}\rho_0\right)$ which violates \eqref{eq:p1}.     
\item  For $k_0<\dfrac{(\Delta_r-\Delta_b)^2}{4\Delta_r\Delta_b}$, \eqref{eq:system at 0} has two disctinct solutions. Since $n_b>n_r>1$, $\Delta_b<\Delta_r$. Hence, only solution \eqref{eq:admissible p} satisfies \eqref{eq:p1}.
\end{itemize}
\end{remark}
\subsubsection{Proof of Theorem  \ref{thm:existenceANDuniquenessOFsolution}} Knowing $k_0=\dfrac{\rho_0}{d_0}<\dfrac{(\Delta_r-\Delta_b)^2}{4\Delta_r\Delta_b}$, let $P=(0,{\bf 0,0,p,p})$ with ${\bf p}=(p_1,p_2,p_3)$ given in \eqref{eq:admissible p}. We will show that the assumptions of Theorem \ref{thm:Main Theorem} are satisfied. Recall the expression of $F_i$ in Definition \ref{def:F_i}.

From Remark \ref{rmk:admissibility of p}, $F(P)={\bf 0}$ and $|p_1|<1$.

We calculate 
$\nabla_{\zeta
}F(P)$. Using \eqref{eqs:identity at 0} and \eqref{eqs:identities for tilde},
\begin{align*}
\dfrac{\partial F_1}{\partial  \zeta_1}(P)&=-\widetilde{F}_{b1}(0,\rho_0,0,0,p_3)=-\rho_0-d_0\left(1-\dfrac{1}{\Delta_b}\dfrac{p_3}{\rho_0}\right)\\
\dfrac{\partial F_2}{\partial \zeta_1}(P)&=M_{r1}(0,\rho_0,0)\widetilde{\Lambda}_{b}(0,\rho_0,0,0,p_3)-\widetilde M_{b1}(0,\rho_0,0,0,p_3)\Lambda_r(0,\rho_0,0)=-\dfrac{1}{\Delta_r}\left(1-\dfrac{1}{\Delta_b}\dfrac{p_3}{\rho_0}\right)\\
\dfrac{\partial F_3}{\partial \zeta_1}(P)&=0.
\end{align*}

From the expression of $\widetilde{M}_{r1}$ in \eqref{eq:aux with tilde}, 
$\dfrac{\partial\widetilde{M}_{r1}}{\partial v_1}(0,\rho_0,0,0,p_3)=0$. Hence, using \eqref{eqs:identity at 0} and \eqref{eq:aux with tilde},
\begin{align*}
\dfrac{\partial F_1}{\partial \zeta_2}(P)&=\dfrac{\partial \widetilde{ F}_{r1}}{\partial v_1}(0,\rho_0,0,0,p_3)=\dfrac{\partial \widetilde{M}_{r1}}{\partial v_1}(0,\rho_0,0,0,p_3)D_r(0,\rho_0,0)+M_{r1}(0,\rho_0,0)\dfrac{\partial {\widetilde D}_r}{\partial v_1}(0,\rho_0,0,p_3)=0\\
\dfrac{\partial F_2}{\partial \zeta_2}(P)&=\dfrac{\partial \widetilde{M}_{r1}}{\partial v_1}(0,\rho_0,0,0,p_3)\Lambda_b(0,\rho_0,0)-M_{b1}(0,\rho_0,0)\dfrac{\partial \widetilde{\Lambda}_{r}}{\partial v_1}(0,\rho_0,0,0,p_3)=0\\
\dfrac{\partial F_3}{\partial \zeta_2}(P)&=1.
\end{align*}

From 
\eqref{eqs:identity at 0} and \eqref{eq:aux with tilde}, we have  $\dfrac{\partial \widetilde{M}_{r1}}{\partial v_2}(0,\rho_0,0,0,p_3)=\dfrac{A_r(\rho_0,0)}{n_r}=\dfrac{-1}{\rho_0\Delta_r}$. Again, using \eqref{eqs:identity at 0} and \eqref{eq:aux with tilde},
\begin{align*}
    \dfrac{\partial F_1}{\partial \zeta_3}(P)&=\dfrac{\partial \widetilde{ F}_{r1}}{\partial v_2}(0,\rho_0,0,0,p_3)=\dfrac{\partial \widetilde{M}_{r1}}{\partial v_2}(0,\rho_0,0,0,p_3)D_r(0,\rho_0,0)+M_{r1}(0,\rho_0,0)\dfrac{\partial \widetilde{D}_r}{\partial v_2}(0,\rho_0,0,p_3)=-\dfrac{d_0}{\rho_0\Delta_r }\\
    \dfrac{\partial F_2}{\partial \zeta_3}(P)&=\dfrac{\partial \widetilde{M}_{r1}}{\partial v_2}(0,\rho_0,0,0,p_3)\Lambda_b(0,\rho_0,0)-M_{b1}(0,\rho_0,0)\dfrac{\partial \widetilde{\Lambda}_{r}}{\partial v_2}(0,\rho_0,0,0,p_3)=-\dfrac{1}{\rho_0\Delta_r\Delta_b}\\
    \dfrac{\partial F_3}{\partial \zeta_3}(P)&=0.
\end{align*}
We conclude that
\begin{equation}\label{eq:nabla zeta F}
    \nabla_{\zeta}F(P)=\begin{pmatrix} -\rho_0-d_0\left(1-\dfrac{1}{\Delta_b}\dfrac{p_3}{\rho_0}\right) & 0 & -\dfrac{d_0}{\rho_0\Delta_r}\\
-\dfrac{1}{\Delta_r}\left(1-\dfrac{1}{\Delta_b}\dfrac{p_3}{\rho_0}\right) & 0 & -\dfrac{1}{\rho_0\Delta_r\Delta_b}\\ 0 & 1 & 0\end{pmatrix}.
\end{equation}
Expanding over the last row, 
\begin{equation*}
    \begin{aligned}
    \det \left(\nabla_{\zeta}F(P)\right) &=-\dfrac{1}{\rho_0\Delta_r\Delta_b}\left(\rho_0+d_0\left(1-\dfrac{1}{\Delta_b}\dfrac{p_3}{\rho_0}\right)\right)  +\dfrac{d_0}{\rho_0\Delta_r}\left( \dfrac{1}{\Delta_r}\left(1-\dfrac{1}{\Delta_b}\dfrac{p_3}{\rho_0}\right)\right) \\
    &=-\dfrac{1}{k_0(\Delta_r\Delta_b)^2}\left[k_0\Delta_r\Delta_b+\Delta_r\left(\Delta_b-\dfrac{p_3}{\rho_0}\right)  - \Delta_b\left(\Delta_b-\dfrac{p_3}{\rho_0}\right)\right].
    \end{aligned}
\end{equation*}
From \eqref{eq:quadratic-equation}, 
\begin{equation}\label{eq:fact-from-quadratic}
k_0\Delta_r\Delta_b +\Delta_r\left(\Delta_b-\frac{p_3}{\rho_0}\right)= \frac{p_3}{\rho_0}\left(\Delta_b-\frac{p_3}{\rho_0}\right).
\end{equation}
Hence, 
\begin{equation}\label{eq:determinant}
\det (\nabla_{\zeta}F(P)) = \dfrac{-1}{k_0(\Delta_r\Delta_b)^2}\left[\frac{p_3}{\rho_0}\left(\Delta_b-\frac{p_3}{\rho_0}\right)  -\Delta_b\left(\Delta_b-\dfrac{p_3}{\rho_0}\right)\right]
=\dfrac{\left(\Delta_b-\dfrac{p_3}{p_0}\right)^2}{k_0(\Delta_r\Delta_b)^2}.   \end{equation}         
From \eqref{eq:p3notDeltab}, $\det (\nabla_{\zeta}F(P))\neq 0$, and so we conclude that $\nabla_{\zeta}F(P)$ is invertible.

It remains to show that the spectral radius of the matrix $\mathcal M:=-[\nabla_{\zeta}F(P)]^{-1}\nabla_{\xi}F(P)$ is smaller than $1$.
From \eqref{eq:nabla zeta F} and \eqref{eq:determinant},
\begin{equation}\label{eq:inverse zeta}
[\nabla_{\zeta}F(P)]^{-1}=\dfrac{k_0(\Delta_r\Delta_b)^2}{\left(\Delta_b-\frac{p_3}{\rho_0}\right)^2}\begin{pmatrix}
    \dfrac{1}{\rho_0\Delta_r\Delta_b} & -\dfrac{d_0}{\rho_0\Delta_r} & 0\\
    0 & 0 & \dfrac{\left(\Delta_b-\frac{p_3}{\rho_0}\right)^2}{k_0\left(\Delta_b\Delta_r\right)^2}\\
    -\dfrac{1}{\Delta_r}\left(1-\dfrac{1}{\Delta_b}\dfrac{p_3}{\rho_0} \right) & \rho_0+d_0\left(1-\dfrac{1}{\Delta_b}\dfrac{p_3}{\rho_0} \right) & 0
\end{pmatrix}.
\end{equation}

Next, we evaluate $\nabla_{\xi}F(P)$. Notice that $F_1,F_2,$ and $F_3$ are independent of $\xi_1$, then $\dfrac{\partial F_1}{\partial \xi_1}=\dfrac{\partial F_2}{\partial \xi_1}=\dfrac{\partial F_3}{\partial \xi_1}=0$.
Moreover, similar to the calculation of $\nabla_{\zeta}F(P)$, we get
\begin{align*}
\dfrac{\partial  F_1}{\partial \xi_2}(P)&=-p_1\dfrac{\partial \widetilde{F}_{b1}}{\partial v_1}(0,\rho_0,0,0,p_3)=0\\
\dfrac{\partial F_2}{\partial \xi_2}(P)&=p_1M_{r1}(0,\rho_0,0)\dfrac{\partial \widetilde{\Lambda}_b}{\partial v_1}(0,\rho_0,0,p_3)-p_1\dfrac{\partial \widetilde{M}_{b1}}{\partial v_1}(0,\rho_0,0,0,p_3)\Lambda_r(0,\rho_0,0)=0\\
\dfrac{\partial F_3}{\partial \xi_2}(P)&=0.
\end{align*}
Finally,
\begin{align*}
    \dfrac{\partial F_1}{\partial \xi_3}(P)&=-p_1\dfrac{\partial \widetilde{F}_{b1}}{\partial v_2}(0,\rho_0,0,0,p_3)=\dfrac{d_0}{\rho_0\Delta_b}p_1\\
    \dfrac{\partial F_2}{\partial \xi_3}(P)&=p_1M_{r_1}(0,\rho_0,0)\dfrac{\partial \widetilde{\Lambda}_b}{\partial v_2}(0,\rho_0,0,0,p_3)-p_1\dfrac{\partial \widetilde{M}_{b1}}{\partial v_2}(0,\rho_0,0,0,p_3)\Lambda_r(0,\rho_0,0)=\dfrac{1}{\rho_0\Delta_r\Delta_b}p_1\\
    \dfrac{\partial F_3}{\partial \xi_3}(P)&=0.
\end{align*}
Hence
\begin{equation}\label{eq:nabla xi}
    \nabla_{\xi}F(P)=\begin{pmatrix} 0 & 0 & \dfrac{d_0}{\rho_0\Delta_b}\\
0 & 0 & \dfrac{1}{\rho_0\Delta_r\Delta_b}\\ 0 & 0 & 0\end{pmatrix} p_1.
\end{equation}

Combining \eqref{eq:inverse zeta} and \eqref{eq:nabla xi}, we obtain that
\begin{equation*}
        \mathcal{M}=-[\nabla_{\zeta}F(P)]^{-1}\nabla_{\xi}F(P) =-\dfrac{k_0(\Delta_r\Delta_b)^2}{\left(\Delta_b-\frac{p_3}{\rho_0}\right)^2}\begin{pmatrix}
        0 & 0 & \dfrac{d_0\left(\Delta_r-\Delta_b\right)}{(\rho_0\Delta_r\Delta_b)^2}\\
        0 &0 & 0\\
        0 & 0 & \dfrac{1}{\Delta_r\Delta_b}
    \end{pmatrix}p_1.
\end{equation*}
The spectral radius for the matrix above is 
\begin{equation*}
    R_{\mathcal{M}}=\left|-\dfrac{k_0(\Delta_r\Delta_b)^2}{\left(\Delta_b-\frac{p_3}{\rho_0}\right)^2}\dfrac{1}{\Delta_r\Delta_b}p_1\right| = \left|\dfrac{k_0\Delta_r\Delta_b}{\left(\Delta_b-\frac{p_3}{\rho_0}\right)^2}p_1\right|.
    \end{equation*}
By making the two consecutive substitutions \eqref{eq:fact-from-quadratic} and \eqref{eq:p1explicit}
\begin{equation*}
    R_{\mathcal{M}}= \left|\dfrac{\left(\frac{p_3}{\rho_0}-\Delta_r\right)\left(\Delta_b-\frac{p_3}{\rho_0}\right)}{\left(\Delta_b-\frac{p_3}{\rho_0}\right)^2}p_1\right| = \left|-\dfrac{\left(\frac{p_3}{\rho_0}-\Delta_r\right)}{\left(\frac{p_3}{\rho_0}-\Delta_b\right)}p_1\right|  =\left|p_1\right|^2 < 1.
\end{equation*}

From Theorem \ref{thm:Main Theorem}, we conclude that  for $k_0<\dfrac{(\Delta_r-\Delta_b)^2}{4\Delta_r\Delta_b}$ there exists  a unique solution $Z$ to \eqref{eq:system} satisfying $Z'(0)={\bf p}$.

It remains to show the symmetry properties \eqref{eq:symmetry of solution}. Define $\hat{Z}(t) = (\hat{Z}_1(t), \hat{Z}_2(t), \hat{Z}_3(t)) = (-Z_1(-t), Z_2(-t), -Z_3(-t))$, and notice that $$\hat{Z}'(t) = (\hat{Z}'_1(t), \hat{Z}'_2(t), \hat{Z}'_3(t)) = (Z_1'(-t), -Z_2'(-t), Z_3'(-t)).$$ 
We have $\hat Z(0)=Z(0)={\bf 0}$ and $\hat Z'(0)=Z'(0)={\bf p}$.
We will show that $\hat{Z}(t)$ also satisfies \eqref{eq:F1explicit}, \eqref{eq:F2explicit}, and \eqref{eq:F3explicit} which by uniqueness implies that $\hat Z=Z.$

In fact, from the formula of $F_1$ in \eqref{eq:F1}, 
{\footnotesize
\begin{align*}
F_1&(t, \hat{Z}(t), \hat{Z}(\hat{Z}_1(t)), \hat{Z}'(t), \hat{Z}'(\hat{Z}_1(t)))\\ 
    &=\tilde {F}_{r1} (t, Z_2(-t) + \rho_0, -Z_3(-t), -Z_2'(-t), Z_3'(-t)) -Z_1'(-t) \tilde {F}_{b1} (-Z_1(-t), Z_2(Z_1(-t))+ \rho_0, -Z_3(Z_1(-t)), -Z_2'(Z_1(-t)), Z_3'(Z_1(-t))).
\end{align*}
}
Hence, the symmetry properties of ${\tilde F}_{r1}$ (and ${\tilde F}_{b1}$) in  \eqref{eq:symmetry tilde} and equation \eqref{eq:F1explicit} give 
{\footnotesize
\begin{align*}
F_1&(t, \hat{Z}(t), \hat{Z}(\hat{Z}_1(t)), \hat{Z}'(t), \hat{Z}'(\hat{Z}_1(t)))
\\
&=\tilde {F}_{r1} (-t, Z_2(-t)+ \rho_0, Z_3(-t), Z_2'(-t), Z_3'(-t)) -Z_1'(-t) \tilde {F}_{b1} (Z_1(-t), Z_2(Z_1(-t))+ \rho_0, Z_3(Z_1(-t)), Z_2'(Z_1(-t)), Z_3'(Z_1(-t))) \\
    &= F_1(-t, Z(-t), Z(Z_1(-t)), Z'(-t), Z'(Z_1(-t))) \\
    &=0.
\end{align*}
}
Similarly, from \eqref{eq:F2}, the symmetry properties \eqref{eq:symmetry nontilde} and \eqref{eq:symmetry tilde}, and equation \eqref{eq:F2explicit}, we get 
\begin{equation*}
    F_2(t, \hat{Z}(t), \hat{Z}(\hat{Z}_1(t)), \hat{Z}'(t), \hat{Z}'(\hat{Z}_1(t))) = F_2(-t, Z(-t), Z(Z_1(-t)), Z'(-t), Z'(Z_1(-t))) = 0.
\end{equation*}
Finally, from \eqref{eq:F3} and \eqref{eq:F3explicit}, {\small
\begin{align*}
            F_3(t, \hat{Z}(t), \hat{Z}(\hat{Z}_1(t)), \hat{Z}'(t), \hat{Z}'(\hat{Z}_1(t))) = -Z_2'(-t)+Z_3(-t)
            =-F_3(-t,Z(-t),Z(Z_1(-t)),Z'(-t),Z'(Z_1(-t)))=0
\end{align*}
}
\qed

We proved in Theorem \ref{thm:existenceANDuniquenessOFsolution} the existence of a solution to system \eqref{eq:system}, with $Z'(0)={\bf p}$ given in \eqref{eq:admissible p}. From Remark \ref{rmk:admissibility of p}, $p_1$ satisfies \eqref{eq:p1}. Then, from Theorem \ref{thm:converse} and Section \ref{subsub:converse}, given $\rho_0, d_0>0$ such that $k_0=\dfrac{\rho_0}{d_0}<\dfrac{(\Delta_r-\Delta_b)^2}{4\Delta_r\Delta_b}$,
there exists a unique lens $(L,S)$ that solves the dichromatic problem in 2D with $(0,\rho_0)\in L$, and $(0,\rho_0+d_0)\in S$.




\subsection{The three dimensional case}\label{subsec:3D}
We now exploit the established solution symmetry to solve the dichromatic problem in three dimensions. Let $Z:=Z(t)$, $t\in D=[-t_0,t_0]$, be the unique solution found in Theorem \ref{thm:existenceANDuniquenessOFsolution}. Let $(L,S)$ be the corresponding lens as constructed in Section \ref{subsub:converse}. The lower face $L$ is parametrized by $\rho(t){\bf x}(t)$ with $\rho(t)=Z_2(t)+\rho_0$. From \eqref{eq:symmetry of solution}, $\rho(-t)=\rho(t)$, and so 
$\rho(-t){\bf x}(-t)=(-\rho(t)\sin t, \rho(t)\cos t)$, which implies that $L$ is symmetric with respect to the vertical axis.

Using  \eqref{eq:symmetry nontilde}, the upper face $S=S_b$ parametrized by 
${\bf f}_b(t)=(f_{b1}(t), f_{b2}(t))={\bf F}_b(t,\rho(t),\rho'(t))$ satisfies the following
    \begin{align*}
        f_{b1}(-t)&=F_{b1}(-t,\rho(-t), \rho'(-t))= F_{b1}(-t, \rho(t),-\rho'(t))=-F_{b1}(t,\rho(t),\rho'(t))
        =-f_{b1}(t)\\
        f_{b2}(-t)&=F_{b2}(-t,\rho(-t), \rho'(-t))= F_{b2}(-t, \rho(t),-\rho'(t))=F_{b2}(t,\rho(t),\rho'(t))
        =f_{b2}(t).
    \end{align*}
Hence, the upper face $S$ is also symmetric with respect to the vertical axis.  Similarly, $S_r$ parametrized by ${\bf f}_r(t)=(f_{r1}(t),f_{r2}(t))=F_r(t,\rho(t),\rho'(t))$ is symmetric with respect to the vertical axis. Letting $\varphi(t)=Z_1(t)$, we have that ${\bf f}_{r}(t)={\bf f}_b(\varphi(t))$.

Revolving both faces of the lens $L$ and $S$ around the vertical axis, we get two surfaces of revolution $\widetilde L$ and $\widetilde S$ in three dimensions such that the lens $(\widetilde L,\widetilde S) $ refracts the cone of dichromatic rays emitted from the origin into the vertical direction 
$(0,0,1)$, see Figure \ref{fig:3D}.

More precisely, for $\theta\in \mathbb R$ and ${\bf a} = (a_1,a_2)\in \mathbb R^2$, define the axial rotation 
$${\bf a}_\theta=R_{\theta}{\bf a}=(a_1 \cos \theta, a_1 \sin \theta, a_2).$$
Let $\Omega\subseteq S^2$ be as follows
$$\Omega=\{{\bf x}_{\theta}(t)=R_{\theta}(\sin t,\cos t):t\in D, \theta\in [0,\pi]\}.$$
We parametrize the lower face of the lens $\widetilde L$ by $\rho(t){\bf x}_{\theta}(t)$ with ${\bf x}_{\theta}(t)\in \Omega$.
Fix $\theta\in [0,\pi]$, let $\Pi_{\theta}$ be the plane containing the $z$-axis (vertical axis) and forming an angle $\theta$ with the $x$-$z$ plane. For each $t$, the normal to ${\widetilde L}$ at the point $\rho(t){\bf x}_{\theta}(t)$ is contained in the plane $\Pi_{\theta}$. Then, by Snell's law, the ray emitted from the origin with frequency $\omega_r$ and direction ${\bf x}_\theta(t)$ is refracted by $\widetilde L$  into the direction $({\bf m}_r)_{\theta}(t)=R_{\theta}{\bf m}_r(t)\in \Pi_{\theta}$; ${\bf m}_r$ given in \eqref{eq:mrM r}. Let $\widetilde S_r$ be the surface parametrized by the vector $$({\bf f}_{r})_{\theta}(t)=\rho(t){\bf x}_{\theta}(t)+d_r(t)({\bf m}_{r})_{\theta}(t)=R_{\theta}{\bf f}_r(t).$$ Again, by Snell's law at the point $({\bf f}_r)_{\theta}(t)$, the ray with direction $({\bf m}_{r})_\theta(t)$ is refracted by $\widetilde S _r$ into the direction $(0,0,1)\in \Pi_{\theta}$. 
Similarly, we define the surface $\widetilde S_b$ parametrized by
$({\bf f}_{b})_{\theta}(t)=R_{\theta}{\bf f}_b(t)$. The lens $(\widetilde L,\widetilde S_b)$ refracts rays with frequency $\omega_b$ and unit direction in $\Omega$ into the direction $(0,0,1)$. 

Notice that
$$({\bf f}_{r})_{\theta}(t)=R_{\theta}{\bf f}_r(t)=R_{\theta}{\bf f}_b(\varphi(t))=({\bf f}_{b})_{\theta}(\varphi(t)).$$ 
Letting $\widetilde S=\widetilde S_b$, we conclude that the lens $(\widetilde L,\widetilde S)$ solves the dichromatic problem for the frequencies $\omega_r$ and $\omega_b$ in three dimensions.

\subsection{Numerical Implementation and results}\label{subsec:numerical}
We implement a Picard fixed-point iteration to construct the solution to the system \eqref{eq:system} with $F_i$ given in Definition \ref{def:F_i}, and we generate the corresponding lens\footnote{The code can be found in the following github repository \url{https://github.com/TarekM678/SRC-Optics-Numerical-Implementation}.}.

For that, we first write the system in the form \eqref{eq:simplified system}. In fact, from \eqref{eq:F3explicit}, $Z_2'(t)=Z_3(t)$, obtaining $H_2$. Replacing in \eqref{eq:F1explicit} and \eqref{eq:F2explicit}, we solve for $Z_1'(t)$ and $Z_3'(t)$ to obtain $H_1$ and $H_3$. We then construct the sequence $$Z^0(t)=t{\bf p}\qquad Z^{n+1}(t)=\int_0^t H(s,Z^{n}(s),{Z^{n}}(Z_1^{n}(s)),({Z^{n}})'(Z_1^{n}(s)))\, ds,$$
where ${\bf p}$ is given in \eqref{eq:admissible p}. The convergence of the iterations for particular values of $k_0$ and for $t_0$ small enough follows from Theorems \ref{thm:Main Theorem} and \ref{thm:existenceANDuniquenessOFsolution}. 

For the numerical implementation, we need to find $Z^n$ at multiple values of $t\in[-t_0,t_0]$ with $t_0$ chosen. We interpolate and evaluate the composition $Z^{n} (Z^{n}_{1}(t))$ and $({Z^n})' (Z^{n}_{1}(t))$ at each iteration. This allows us to evaluate $H$ in the integrand and find $Z^{n+1}(t)$. 

Choosing a small tolerance $\delta_0$, we define $Z(t) := Z^{n_0}(t)$ such that $$\max_{t}\frac{|Z^{n_0}(t) -Z^{n_0-1}(t) |}{|Z^{n_0-1}(t)|}< \delta_0.$$ We then take $\rho(t)=Z_2(t)+\rho_0$, and implement Snell's law to calculate ${\bf m_r}$ and ${\bf m}_b$. Now, we construct the corresponding faces ${\bf f}_r$ and ${\bf f}_b$.   
For $n_r = 1.14$ and $ n_b =1.866$, we obtain  Figure \ref{figure code}.

\begin{figure}[H]
    \centering
\includegraphics[width=5cm, height=10cm]{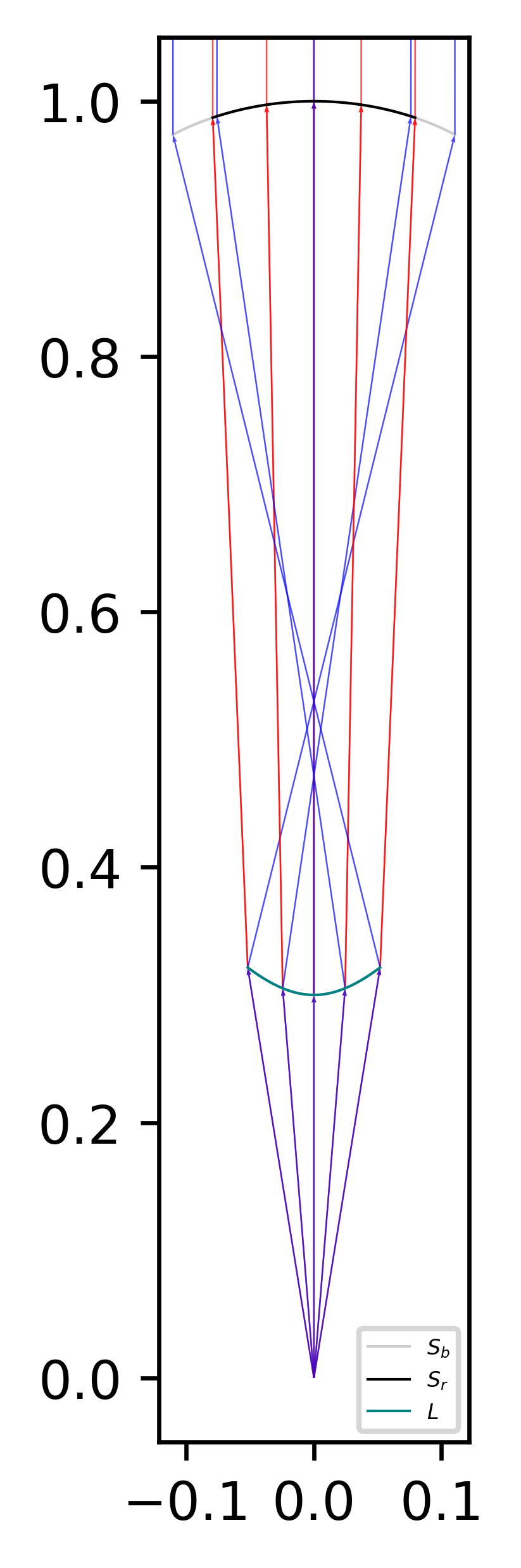}
    \caption{Lens solution with $\rho_0 =0.3 , d_0=0.7,t_0 = 0.16, \delta_0 =10^{-12}$}
\label{figure code}
\end{figure}
We also implemented a Runge-Kutta algorithm (RK4) to solve the system of FDEs which allowed faster convergence, see Table \ref{Table}
\begin{table}[H]
{\small  \begin{tabular}{|c|c|c|c|c|c|}
        \hline
        Schemes & { \footnotesize $\rho_0 = 0.5, d_0= 1.2$} & {\footnotesize $\rho_0 = 0.5, d_0= 1.5$} & { \footnotesize$\rho_0 = 0.5, d_0= 2$} & { \footnotesize$\rho_0 = 0.5, d_0= 2.5$}\\
        \hline
        Picard Lindel\"of & $8.4s$ & $7.8s$ & $8.1s$ & $9.2s$ \\
        \hline
        RK4 & $6.4s$ & $6.7s$ & $6.8s$ & $6.6s$ \\
        \hline
        \hline
    \end{tabular}}
    \caption{Run-times for RK4 and Picard Lindel\"of with $\delta_0 = 10^{-10}$. 
    }\label{Table}
\end{table}

We notice that for large angle $t_0$ the lens will suffer from self intersections or cusps emphasizing that our solution could only be local, see Figure \ref{fig:cusp}.
\begin{figure}[H]
    \centering
\includegraphics[width=0.5\linewidth]{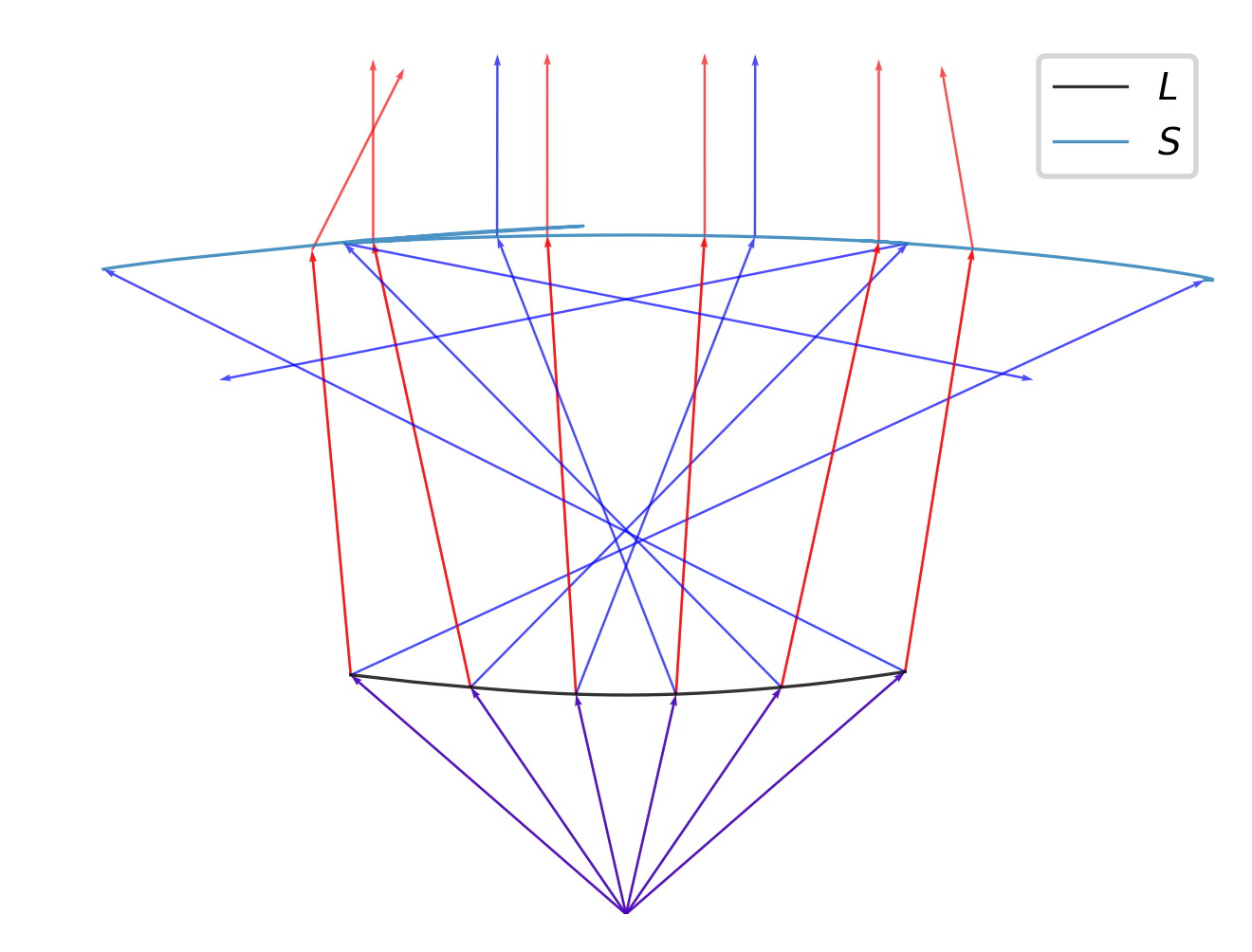}
    \caption{Obtained lens for $\rho_0=0.1,d_0 = 0.21$, and $t_0=0.2$. 
    }
    \label{fig:cusp}
\end{figure}
\noindent We then investigated, using the Picard scheme, the largest $t_0$ for which the algorithm converges for different values of $k_0$, and noticed that as $k_0$ approaches the limiting value $\dfrac{(\Delta_r-\Delta_b)^2}{4\Delta_r\Delta_b}$, $t_0$ becomes close to zero, see Figure \ref{fig:k_0different}.
\begin{figure}[H]
    \centering
    \includegraphics[width=0.5\linewidth]{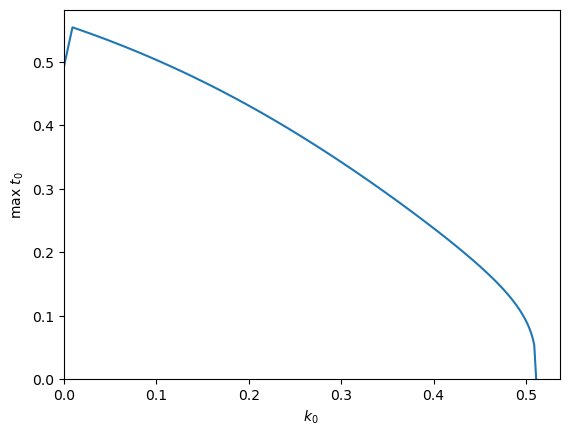}
    \caption{Maximal angle $t_0$ with respect to $k_0$.
    }
    \label{fig:k_0different}
\end{figure}
\noindent For $k_0$ at the limiting value, the Picard scheme was unable to construct a lens for $t_0 > 3.218*10^{-5}$, whereas the RK4 scheme was able to construct a lens for $t_0 = 0.04$.

Finally, we revolve the 2D lens in Figure \ref{figure code} to obtain the three dimensional lens in Figure \ref{fig:3D}.
\begin{figure}[H]
    \centering
    \includegraphics[width=0.42\linewidth]{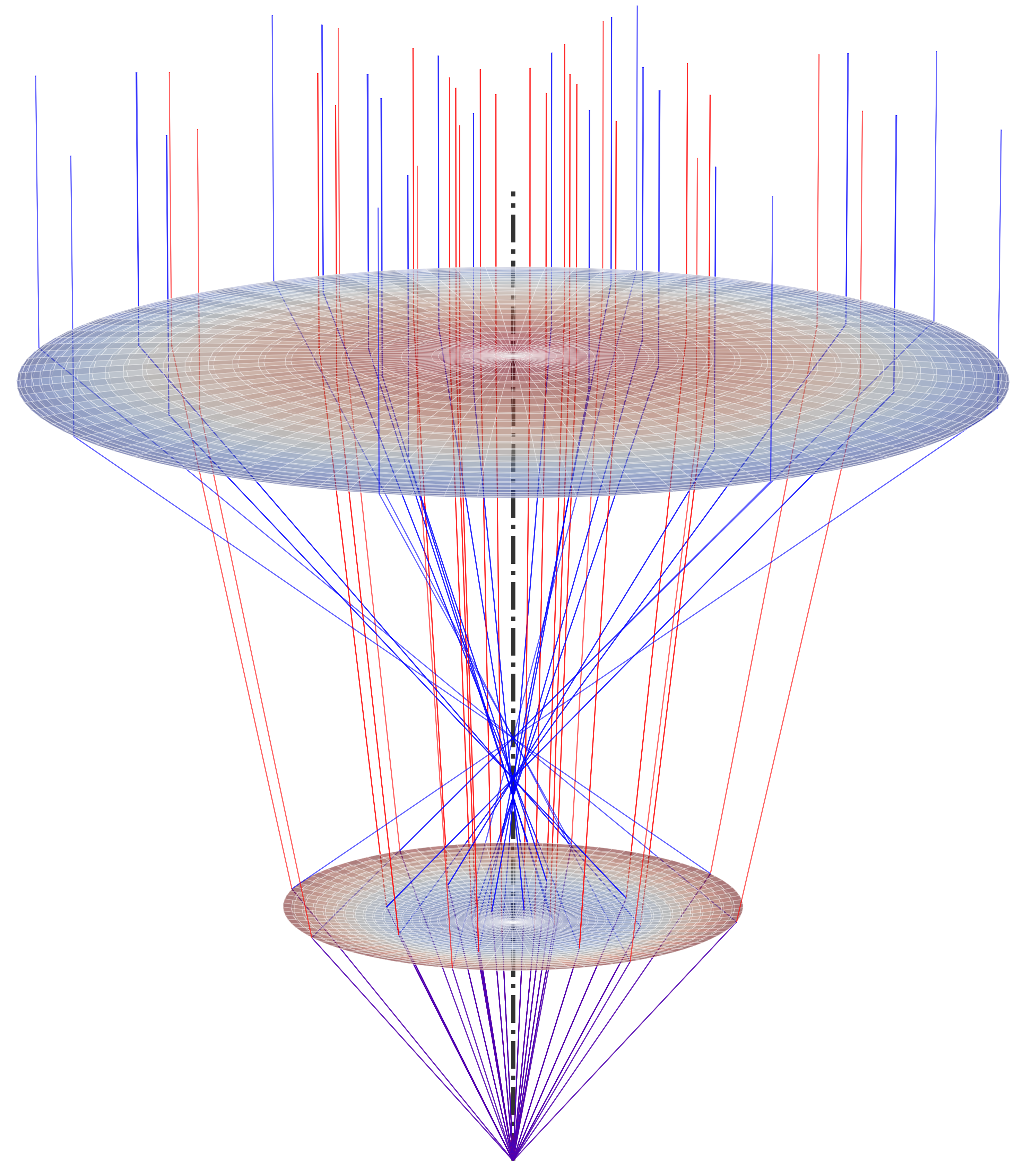}
    \caption{3D lens solution with $\rho_0 =0.3 , d_0=0.7,t_0 = 0.16, \delta_0 =10^{-12}$.}
    \label{fig:3D}
\end{figure}

\section*{Acknowledgments}
All authors were partially supported by the University Research Board, Grants 104631 from the American University of Beirut, and by the CAMS-Department of Mathematics Summer Research Program.  

\bibliographystyle{unsrt}
\bibliography{References}

@book{BornWolf1999,
  author    = {Max Born and Emil Wolf},
  title     = {Principles of Optics: Electromagnetic Theory of Propagation, Interference and Diffraction of Light},
  edition   = {7},
  publisher = {Cambridge University Press},
  address   = {Cambridge},
  year      = {1999},
  isbn      = {978-0521642224},
}

@article{CaffarelliOliker2008,
  author  = {Luis A. Caffarelli and Vladimir I. Oliker},
  title   = {Weak Solutions of One Inverse Problem in Geometric Optics},
  journal = {Journal of Mathematical Sciences},
  volume  = {154},
  number  = {1},
  pages   = {39--49},
  year    = {2008},
  month   = oct,
  doi     = {10.1007/s10958-008-9152-x},
}

@article{Chen2018,
  author  = {Wei Ting Chen and Alexander Y. Zhu and Vyshakh Sanjeev and Mohammadreza Khorasaninejad and Zhujun Shi and Eric Lee and Federico Capasso},
  title   = {A Broadband Achromatic Metalens for Focusing and Imaging in the Visible},
  journal = {Nature Nanotechnology},
  volume  = {13},
  number  = {3},
  pages   = {220--226},
  year    = {2018},
  doi     = {10.1038/s41565-017-0034-6},
}

@article{FriedmanMcLeod1987,
  author  = {Avner Friedman and John B. McLeod},
  title   = {Optimal Design of an Optical Lens},
  journal = {Archive for Rational Mechanics and Analysis},
  volume  = {99},
  number  = {2},
  pages   = {147--164},
  year    = {1987},
  doi     = {10.1007/BF00275875},
}

@article{guan1998monge,
  title={On a {M}onge-{A}mp{\`e}re equation arising in geometric optics},
  author={Guan, Pengfei and Wang, Xu-Jia},
  journal={Journal of Differential Geometry},
  volume={48},
  number={2},
  pages={205--223},
  year={1998},
  publisher={Lehigh University}
}

@article{Gutierrez:14,
author = {Cristian E. Guti\'{e}rrez and Ahmad Sabra},
journal = {J. Opt. Soc. Am. A},
number = {4},
pages = {891--899},
publisher = {Optica Publishing Group},
title = {Design of pairs of reflectors},
volume = {31},
month = {Apr},
year = {2014},
doi = {10.1364/JOSAA.31.000891},
}

@article{SIAMGUTSAB,
author = {Guti\'{e}rrez, Cristian E. and Sabra, Ahmad},
title = {Aspherical Lens Design and Imaging},
journal = {SIAM Journal on Imaging Sciences},
volume = {9},
number = {1},
pages = {386-411},
year = {2016},
doi = {10.1137/15M1030807},
}

@article{Gutierrez:18,
author = {Cristian E. Guti\'{e}rrez and Luca Pallucchini},
journal = {J. Opt. Soc. Am. A},
number = {9},
pages = {1523--1531},
publisher = {Optica Publishing Group},
title = {Reflection and refraction problems for metasurfaces related to {M}onge-{A}mp{\`e}re equations},
volume = {35},
month = {Sep},
year = {2018},
doi = {10.1364/JOSAA.35.001523},

}

@article{Gutierrez2013,
  author    = {Guti\'{e}rrez, Cristian E.},
  title     = {Aspherical Lens Design},
  journal   = {Journal of the Optical Society of America A},
  year      = {2013},
  volume    = {30},
  number    = {9},
  pages     = {1719},
  doi       = {10.1364/JOSAA.30.001719},
  publisher = {Optica Publishing Group}
}

@article{GutierrezHuang2014,
  author  = {Cristian E. Guti{\'e}rrez and Qingbo Huang},
  title   = {The Near Field Refractor},
  journal = {Annales de l'Institut Henri Poincar\'e C, Analyse Non Lin\'eaire},
  volume  = {31},
  number  = {4},
  pages   = {655--684},
  year    = {2014},
  doi     = {10.1016/j.anihpc.2013.07.001},
}

@article{GutierrezSabra2021,
  author  = {Cristian E. Guti{\'e}rrez and Ahmad Sabra},
  title   = {Chromatic Aberration in Metalenses},
  journal = {Advances in Applied Mathematics},
  volume  = {124},
  pages   = {102134},
  year    = {2021},
  doi     = {10.1016/j.aam.2020.102134},
}

@incollection{GutierrezSabra2020,
  author    = {Cristian E. Guti{\'e}rrez and Ahmad E. Sabra},
  title     = {On the Existence of Dichromatic Single Element Lenses},
  booktitle = {Advances in Harmonic Analysis and Partial Differential Equations},
  series    = {Contemporary Mathematics},
  volume    = {748},
  pages      = {99--145},
  publisher = {American Mathematical Society},
  address   = {Providence, RI},
  year      = {2020},
  doi       = {10.1090/conm/748/15057},
}

@article{GutierrezHuang2009,
  author  = {Cristian E. Guti{\'e}rrez and Qingbo Huang},
  title   = {The Refractor Problem in Reshaping Light Beams},
  journal = {Archive for Rational Mechanics and Analysis},
  volume  = {193},
  number  = {2},
  pages   = {423--443},
  year    = {2009},
  doi     = {10.1007/s00205-008-0165-x},
}

@book{HaleLunel1993,
  author    = {Jack K. Hale and Sjoerd M. Verduyn Lunel},
  title     = {Introduction to Functional Differential Equations},
  series    = {Applied Mathematical Sciences},
  volume    = {99},
  publisher = {Springer},
  address   = {New York},
  year      = {1993},
  doi       = {10.1007/978-1-4612-4342-7},
}

@article{Hua2017,
  author  = {Hong Hua and Yuzuru Tsuboi and Yuki Yamanaka and Hiroyuki Oshika},
  title   = {Achromatic Doublet Intraocular Lens for Full Aberration Correction},
  journal = {Biomedical Optics Express},
  volume  = {8},
  number  = {7},
  pages   = {3240--3248},
  year    = {2017},
  doi     = {10.1364/BOE.8.003240},
  pmid    = {28663881},
  pmcid   = {PMC5480488},
}

@book{JenkinsWhite2001,
  author    = {Francis A. Jenkins and Harvey E. White},
  title     = {Fundamentals of Optics},
  edition   = {4},
  publisher = {McGraw-Hill},
  address   = {New York},
  year      = {2001},
  isbn      = {978-0072561913},
}

@article{MerigotThibert2021,
  author  = {Quentin M{\'e}rigot and Boris Thibert},
  title   = {Mirrors, Lenses and {M}onge-{A}mp{\`e}re Equations},
  journal = {European Mathematical Society Magazine},
  volume  = {120},
  year    = {2021},
  pages   = {16--28},
  doi     = {10.4171/MAG/21},
}

@inproceedings{Oliker2019,
  author    = {Vladimir Oliker},
  title     = {Freeform Optics for Illumination and Supporting Quadric Method (SQM)},
  booktitle = {Optical Design and Fabrication 2019 (Freeform, OFT)},
  series    = {OSA Technical Digest},
  publisher = {Optica Publishing Group},
  year      = {2019},
  pages      = {FT1B.6},
  doi       = {10.1364/FREEFORM.2019.FT1B.6},
}

@article{rogers1988existence,
  title={Existence, uniqueness, and construction of the solution of a system of ordinary functional differential equations, with application to the design of perfectly focusing symmetric lenses},
  author={Rogers, Joel  CW},
  journal={IMA journal of applied mathematics},
  volume={41},
  number={2},
  pages={105--134},
  year={1988},
  publisher={Oxford University Press}
}

@book{Serre,
    author = {Denis Serre},
    title ={Matrices: Theory and applications},
    publisher = {Springer},
    year = {2010}
}

@article{SunChuTien2009,
  author  = {Wen-Shing Sun and Chien-Hsun Chu and Chuen-Lin Tien},
  title   = {Well-Chosen Method for an Optimal Design of Doublet Lens Design},
  journal = {Optics Express},
  volume  = {17},
  number  = {3},
  pages   = {1414--1428},
  year    = {2009},
  doi     = {10.1364/OE.17.001414},
}

@article{van1992lens,
  title={A lens focusing light at two different wavelengths},
  author={Van-Brunt, Bruce and Ockendon, John R},
  journal={Journal of mathematical analysis and applications},
  volume={165},
  number={1},
  pages={156--179},
  year={1992},
  publisher={Elsevier}
}

@article{van1994mathematical,
  title={Mathematical possibility of certain systems in geometrical optics},
  author={Van-Brunt, Bruce},
  journal={Journal of the Optical Society of America A},
  volume={11},
  number={11},
  pages={2905--2914},
  year={1994},
  publisher={Optical Society of America}
}

@inproceedings{Verma2026,
  author    = {Sanjana Verma and Martijn J. H. Anthonissen and Jan H. M. ten Thije Boonkkamp and Wilbert L. IJzerman and Lisa Kusch},
  title     = {Inverse Methods for Freeform Imaging Design},
  booktitle = {AIP Conference Proceedings},
  volume     = {3489},
  number     = {1},
  pages      = {280005},
  year       = {2026},
  editor     = {Theodore E. Simos and Charalambos Tsitouras},
  publisher  = {AIP Publishing},
  doi        = {10.1063/5.0328468},
}
\end{document}